\documentclass[a4paper,11pt]{article}

\usepackage[utf8]{inputenc}
\usepackage[english,activeacute]{babel}
\usepackage{amsmath,amsthm,amsfonts,amssymb}
\usepackage{mathpartir}
\usepackage{ stmaryrd }
\usepackage{graphics}
\usepackage{enumerate}
\usepackage{mathrsfs}

\input{xy}
\usepackage[all]{xy}

\theoremstyle{plain}
\newtheorem{thm}{Theorem}[subsection]
\newtheorem{cor}[thm]{Corollary}
\newtheorem{lemma}[thm]{Lemma}
\newtheorem{proposition}[thm]{Proposition}
\newtheorem{rmk}[thm]{Remark}
\theoremstyle{plain}
\newtheorem{defs}[thm]{Definition}
\theoremstyle{remark}

\input{diagxy}
\xyoption{curve}

\newbox\anglebox 
\setbox\anglebox=\hbox{\xy \POS(75,0)\ar@{-} (0,0) \ar@{-} (75,75)\endxy}
 
\newbox\angleboxr 
\setbox\angleboxr=\hbox{\xy \POS(0,0)\ar@{-} (0,75) \ar@{-} (75,0)\endxy}
 
\newbox\sanglebox 
\setbox\sanglebox=\hbox{\xy \POS(50,0)\ar@{-} (0,0) \ar@{-} (50,50)\endxy}
 
\newbox\sangleboxr 
\setbox\sangleboxr=\hbox{\xy \POS(0,0)\ar@{-} (0,50) \ar@{-} (50,0)\endxy}
 
\newbox\sangleboxf 
\setbox\sangleboxf=\hbox{\xy \POS(50,50)\ar@{-} (50,0) \ar@{-} (0,50)\endxy}
 
\newbox\angleboxf 
\setbox\angleboxf=\hbox{\xy \POS(75,75)\ar@{-} (75,0) \ar@{-} (0,75)\endxy}
 
\newbox\sangleboxfr 
\setbox\sangleboxfr=\hbox{\xy \POS(0,50)\ar@{-} (50,50) \ar@{-} (0,0)\endxy}
 
\newbox\angleboxfr 
\setbox\angleboxfr=\hbox{\xy \POS(0,75)\ar@{-} (75,75) \ar@{-} (0,0)\endxy}

\newcommand{\cat}[1]{\ensuremath{\mathcal{#1}}}

\newcommand{\synt}[2]{\ensuremath{\mathcal{#1}_{{#2}}}}
\newcommand{\alg}[1]{\ensuremath{\mathbf{#1}}}

\newcommand{\Sets}{\ensuremath{\mathbf{Set}}}

\newcommand{\fins}[1]{\exists {#1}\mathpunct .}

\newcommand{\theory}{\ensuremath{\mathbb{T}}}

\title{INFINITARY FIRST-ORDER CATEGORICAL LOGIC}
\author{Christian Esp\'indola}

\begin{document}
\date{}
\maketitle

\begin{abstract}
We present a unified categorical treatment of completeness theorems for several classical and intuitionistic infinitary logics with a proposed axiomatization. This provides new completeness theorems and subsumes previous ones by G\"odel, Kripke, Beth, Karp, Joyal, Makkai and Fourman/Grayson. As an application we prove, using large cardinals assumptions, the disjunction and existence properties for infinitary intuitionistic first-order logics.
\end{abstract}

\tableofcontents

\section{Introduction}

A very natural and fundamental problem in mathematical logic is the determination of the precise set of axioms that are enough to guarantee the completeness property of a certain logic. That is, to axiomatize the logic in question so that there is a relation between semantic validity and syntactic provability. For classical and intuitionistic first-order and propositional logic the answers are known. As soon as the expressive power of the logic increases, proving a completeness theorem usually requires a stronger metatheory, often relying on set-theoretical assumptions. For example, completeness theorems for infinitary classical logic are known to depend on the addition of large cardinal axioms in the metatheory. In this work, we plan to consider infinitary intuitionistic logics in full generality, and to provide axiomatizations and completeness theorems. Unlike the common techniques used to establish completeness, the main insight we will follow is to establish a general framework within categorical logic that results adequate for the development of infinitary languages. Building on previous work by Karp \cite{karp} and Makkai \cite{makkai}, we will see that this setting is precisely that of infinitary first-order categorical logic. The unifying power of category-theoretic language provided the right means to subsume and amalgamate the existing completeness theorems in a unique framework, where new completeness results branch out, combining aspects of categorical logic, sheaf theory, model theory and set theory.

Completeness theorems for classical propositional logic were developed by Hilbert and Ackermann, and extended to the classical first-order case by G\"odel in his doctoral dissertation \cite{godel}. For intuitionistic logic the completeness in terms of the semantic of possible worlds was introduced by Kripke in \cite{kripke}. Meanwhile, completeness of infinitary logics was obtained by Karp in the monograph \cite{karp}, where systems for classical infinitary propositional and first-order logic have been described and studied extensively with Hilbert-type systems (see also \cite{mt} for a related development with Gentzen's sequents). Nadel developed infinitary intuitionistic propositional logic for countable many conjunctions/disjunctions in \cite{nadel}, proving completeness with respect to the infinitary version of Kripke semantics. On the other hand, Makkai considered infinitary regular theories in \cite{makkai} together with a corresponding completeness theorem. 

Makkai's work has a great deal of categorical logic, a subject prominent since the seventies, when the school of Montr\'eal introduced the discipline exploiting the techniques of category theory to study logical theories. The philosophy of theories as categories had reached its climax in the work of Makkai and Reyes \cite{mr} (1977), where the authors discuss an extensive treatment of categorical logic for first-order theories, and include some infinitary versions. We intend to provide the jump from Makkai's regular case in \cite{makkai} to the first-order case by means of a completeness theorem for infinitary coherent logic, which together with a generalization of an unpublished theorem of A. Joyal will allow us to derive completeness for infinitary intuitionistic propositional and first-order logics, in terms of infinitary Kripke semantics, as well as sheaf and categorical models. Unlike classical infinitary logics, whose related completeness results have been known for decades, the main difficulty in studying infinitary intuitionistic logics is the huge variety of non-equivalent formulas that one can obtain. This, as we will show, suggests the introduction of large cardinal axioms appropriate to handle this unexpected richness that one gets by dropping the excluded middle.

Instead of relying on Henkin's method of adding constants, which encounters some problems in the infinitary case, the way that we will prove completeness relies on manipulating sheaf models for infinitary logics (these are explained for example in \cite{bj}). The expert will however find that the methods here developed are a refinement of those of Henkin, in the sense that instead of introducing witnessing constants for existential sentences, they will indirectly appear naturally as a result of forcing statements in our way to building models for the theory. The key contribution is the identification of the transfinite transitivity property as the correct categorical counterpart of the axiomatic treatment proposed. This property is a generalization of the transitivity property for Grothendieck topologies, and plays an essential r\^ole in the completeness proof.

The structure of this work is as follows. In the first part we introduce a system for infinitary intuitionistic logic and we prove that the addition of the axiom of excluded middle results in a system equivalent to Karp's classical system. The second part deals with the categorical counterpart of this type of logic, and introduces many of the developments of categorical logic generalized to the infinitary case; it also contains completeness proofs for infinitary intuitionistic logic in terms of sheaf and more general categorical models. The third part contains the proof of completeness of infinitary intuitionistic logic with respect to Kripke semantics (a generalization of Joyal's theorem) relying on a large cardinal hypothesis and on a completeness theorem for infinitary coherent logic. Along the way we prove an infinitary version of completeness for Beth semantics. We derive also Karp completeness theorem from ours, as well as Makkai's and Fourman/Grayson's. When restricting to the finitary case, these theorems reduce to the completeness theorems of G\"odel, Kripke, Beth, and the completeness of the regular and coherent fragment of first-order logic, providing thus new proofs for all of them.

\subsection{Karp's system}

Infinitary languages $\mathcal{L}_{\kappa, \lambda}$ are defined according to the length of infinitary conjunctions/disjunctions as well as quantification it allows. In that way, assuming a supply of $\kappa$ variables to be interpreted as ranging over a nonempty domain, one includes in the inductive definition of formulas an infinitary clause for conjunctions and disjunctions, namely, whenever the ordinal indexed sequence $A_0, ..., A_{\delta}, ...$ of formulas has length less than $\kappa$, one can form the infinitary conjunction/disjunction of them to produce a formula. Analogously, whenever an ordinal indexed sequence of variables has length less than $\lambda$, one can introduce one of the quantifiers $\forall$ or $\exists$ together with the sequence of variables in front of a formula to produce a new formula. One also stipulates that $\kappa$ be a regular cardinal, so that the length of any well-formed formula is less than $\kappa$ itself. 

It is then a natural question to ask for which of these infinitary languages one can provide a notion of provability for which a form of completeness theorem can be proven, in terms, for example, of the obvious Tarskian semantics associated to them. In \cite{karp}, Karp proves completeness theorems for the classical logic $\mathcal{L}_{\kappa, \kappa}$, for inaccessible $\kappa$, within a Hilbert-style system including the distributivity and the dependent choice axioms. These axioms consist of the following schemata:

\begin{enumerate}
 \item $A \to [B \to A]$
 \item $[A \to [B \to C] \to [[A \to B] \to [A \to C]]]$
 \item $[\neg B \to \neg A] \to [A \to B]$
 \item $[\bigwedge_{i<\alpha} [A \to A_i]] \to [A \to \bigwedge_{i<\alpha} A_i]$
 \item $[\bigwedge_{i<\alpha}  A_i] \to A_j$
 \item $[\forall \mathbf{x} [A \to B] \to [A \to \forall \mathbf{x} B]]$
 
 provided no variable in $\mathbf{x}$ occurs free in $A$;
 
 \item $\forall \mathbf{x} A \to S_f(A)$
 
 where $S_f(A)$ is a substitution based on a function $f$ from $\mathbf{x}$ to the terms of the language;
 
 \item Equality axioms:
 
 \begin{enumerate}
 \item $t=t$
 \item $[\bigwedge_{i<\alpha}t_i=t'_i] \to [\phi(t_0, ..., t_\xi, ...)=\phi(t'_0, ..., t'_\xi, ...)]$
 \item $[\bigwedge_{i<\alpha}t_i=t'_i] \to [P(t_0, ..., t_\xi, ...) \to P(t'_0, ..., t'_\xi, ...)]$
 
 for each $\alpha<\kappa$, where $t, t_i$ are terms and $\phi$ is a function symbol of arity $\alpha$ and $P$ a relation symbol of arity $\alpha$;
 \end{enumerate}
 
 \item Classical distributivity axiom\footnote{Throughout this work the notation $\alpha^{\beta}$ for ordinals $\alpha, \beta$ will always denote the set of functions $f: \beta \to \alpha$, and should not be confused with ordinal exponentiation.}:
 
 $$\bigwedge_{i<\gamma}\bigvee_{j<\gamma} \psi_{ij} \to \bigvee_{f \in \gamma^{\gamma}}\bigwedge_{i<\gamma} \psi_{if(i)}$$
 
 \item Classical dependent choice axiom:
 
 $$\bigwedge_{\alpha < \gamma} \forall_{\beta < \alpha}\mathbf{x}_{\beta} \exists \mathbf{x}_{\alpha} \psi_{\alpha} \to \exists_{\alpha < \gamma} \mathbf{x}_{\alpha} \bigwedge_{\alpha < \gamma} \psi_{\alpha}$$
 
 provided the sets $\mathbf{x}_{\alpha}$ are pairwise disjoint and no variable in $\mathbf{x}_{\alpha}$ is free in $\psi_{\beta}$ for $\beta<\alpha$.
 
\end{enumerate}

The inference rules are modus ponens, conjunction introduction and generalization.

In the same way that for finitary languages proofs are finitary objects, the right metatheory to study formal proofs of infinitary languages is that of sets hereditarily of cardinal less than $\kappa$. Similarly, G\"odel numberings of finitary formulas can be generalized to the infinitary case if one uses (not necessarily finite) ordinal numbers (see \cite{karp}), by considering one-to-one functions from the symbols of the language into $\kappa$. It is then possible to consider G\"odel numbers of formulas and prove that they correspond to those sets hereditarily of cardinal less than $\kappa$ that satisfy a precise ordinary predicate in a certain metalanguage. Moreover, G\"odel numbers of provable formulas must satisfy a precise provability predicate in such metalanguage.

The development of \cite{karp} is classical, that is, the infinitary systems considered formalize infinitary classical logic. Intuitionistic systems of infinitary logic using countable many conjunctions and disjunctions was studied in \cite{nadel}. Our purpose here is to study systems for the intuitionistic general case, together with corresponding completeness theorems.

\subsection{Infinitary first-order systems}

Let $\kappa$ be an inaccessible cardinal (we consider $\omega$ to be inaccessible as well, so that our account embodies in particular the finitary case). The syntax of intuitionistic $\kappa$-first-order logics $\mathcal{L}_{\kappa, \kappa}$ consists of a (well-ordered) set of sorts and a set of function and relation symbols, these latter together with the corresponding type, which is a subset with less than $\kappa$ many sorts. Therefore, we assume that our signature may contain relation and function symbols on $\gamma<\kappa$ many variables, and we suppose there is a supply of $\kappa$ many fresh variables of each sort. Terms and atomic formulas are defined as usual, and general formulas are defined inductively according to the following:

\begin{defs} If $\phi, \psi, \{\phi_{\alpha}: \alpha<\gamma\}$ (for each $\gamma<\kappa$) are formulas of $\mathcal{L}_{\kappa, \kappa}$, the following are also formulas: $\bigwedge_{\alpha<\gamma}\phi_{\alpha}$, $\bigvee_{\alpha<\gamma}\phi_{\alpha}$, $\phi \to \psi$, $\forall_{\alpha<\gamma} x_{\alpha} \phi$ (also written $\forall \mathbf{x}_{\gamma} \phi$ if $\mathbf{x}_{\gamma}=\{x_{\alpha}: \alpha<\gamma\}$), $\exists_{\alpha<\gamma} x_{\alpha} \phi$ (also written $\exists \mathbf{x}_{\gamma} \phi$ if $\mathbf{x}_{\gamma}=\{x_{\alpha}: \alpha<\gamma\}$).
\end{defs}
 
The inductive definition of formulas allows to place them in hierarchies or levels up to $\kappa$. Formulas in a successor level are built using the clauses of the definition from formulas in the previous level, while at limit levels one takes the union of all formulas in all levels so far defined. Proofs by induction on the complexity of the formula are proofs by transfinite induction on the least level of the formulas.

The infinitary systems that we will use in categorical logic for our purposes have all the rules of finitary first-order logic, except that in the case of $\mathcal{L}_{\kappa, \kappa}$ we allow infinite sets of variables as contexts of the sequents. Since the variables of each sort are assumed to in correspondence with (the elements of) $\kappa$, each subset of variables comes with an inherited well-order, which we will assume as given when we quantify over sets of variables. There are two special types of formulas that one can consider. One is the class of $\kappa$-\emph{regular} formulas (see \cite{makkai}), which are those build of atomic formulas, $\kappa$-conjunctions and $\kappa$-existential quantification. Adding $\kappa$-disjunction  results in the class of $\kappa$-\emph{coherent} formulas. We shall introduce both of these in more detail later. 

We use sequent style calculus to formulate the axioms of first-order logic, as can be found, e.g., in \cite{johnstone}, D1.3. The system for $\kappa$-first order logic is described below. Its key feature and the difference with the system of Karp essentially resides, besides being an intuitionistic system, in the introduction of the transfinite transitivity rule, which, as we shall see, is an intuitionistic way of merging the classical distributivity and dependent choice axioms. The intuitive meaning of this rule will be further explained after the following:

\begin{defs}\label{sfol}
 The system of axioms and rules for $\kappa$-first-order logic consists of

\begin{enumerate}
 \item Structural rules:
 \begin{enumerate}
 \item Identity axiom:
\begin{mathpar}
\phi \vdash_{\mathbf{x}} \phi 
\end{mathpar}
\item Substitution rule:
\begin{mathpar}
\inferrule{\phi \vdash_{\mathbf{x}} \psi}{\phi[\mathbf{s}/\mathbf{x}] \vdash_{\mathbf{y}} \psi[\mathbf{s}/\mathbf{x}]} 
\end{mathpar}

where $\mathbf{y}$ is a string of variables including all variables occurring in the string of terms $\mathbf{s}$.
\item Cut rule:
\begin{mathpar}
\inferrule{\phi \vdash_{\mathbf{x}} \psi \\ \psi \vdash_{\mathbf{x}} \theta}{\phi \vdash_{\mathbf{x}} \theta} 
\end{mathpar}
\end{enumerate}

\item Equality axioms:

\begin{enumerate}
\item 

\begin{mathpar}
\top \vdash_{x} x=x 
\end{mathpar}

\item 

\begin{mathpar}
(\mathbf{x}=\mathbf{y}) \wedge \phi \vdash_{\mathbf{z}} \phi[\mathbf{y}/\mathbf{x}]
\end{mathpar}

where $\mathbf{x}$, $\mathbf{y}$ are contexts of the same length and type and $\mathbf{z}$ is any context containing $\mathbf{x}$, $\mathbf{y}$ and the free variables of $\phi$.
\end{enumerate}

\item Conjunction axioms and rules:

$$\bigwedge_{i<\gamma} \phi_i \vdash_{\mathbf{x}} \phi_j$$

\begin{mathpar}
\inferrule{\{\phi \vdash_{\mathbf{x}} \psi_i\}_{i<\gamma}}{\phi \vdash_{\mathbf{x}} \bigwedge_{i<\gamma} \psi_i}
\end{mathpar}

for each cardinal $\gamma<\kappa$.

\item Disjunction axioms and rules:

$$\phi_j \vdash_{\mathbf{x}} \bigvee_{i<\gamma} \phi_i$$

\begin{mathpar}
\inferrule{\{\phi_i \vdash_{\mathbf{x}} \theta\}_{i<\gamma}}{\bigvee_{i<\gamma} \phi_i \vdash_{\mathbf{x}} \theta}
\end{mathpar}

for each cardinal $\gamma<\kappa$.

\item Implication rule:

\begin{mathpar}
\mprset{fraction={===}}
\inferrule{\phi \wedge \psi \vdash_{\mathbf{x}} \theta}{\phi \vdash_{\mathbf{x}} \psi \to \theta}
\end{mathpar}

\item Existential rule:
\begin{mathpar}
\mprset{fraction={===}}
\inferrule{\phi \vdash_{\mathbf{x} \mathbf{y}} \psi}{\exists \mathbf{y}\phi \vdash_{\mathbf{x}} \psi}
\end{mathpar}

where no variable in $\mathbf{y}$ is free in $\psi$.

\item Universal rule:
\begin{mathpar}
\mprset{fraction={===}}
\inferrule{\phi \vdash_{\mathbf{x} \mathbf{y}} \psi}{\phi \vdash_{\mathbf{x}} \forall \mathbf{y}\psi}
\end{mathpar}

where no variable in $\mathbf{y}$ is free in $\phi$.

\item Transfinite transitivity rule:

\begin{mathpar}
\inferrule{\phi_{f} \vdash_{\mathbf{y}_{f}} \bigvee_{g \in \gamma^{\beta+1}, g|_{\beta}=f} \exists \mathbf{x}_{g} \phi_{g} \\ \beta<\gamma, f \in \gamma^{\beta} \\\\ \phi_{f} \dashv \vdash_{\mathbf{y}_{f}} \bigwedge_{\alpha<\beta}\phi_{f|_{\alpha}} \\ \beta < \gamma, \text{ limit }\beta, f \in \gamma^{\beta}}{\phi_{\emptyset} \vdash_{\mathbf{y}_{\emptyset}} \bigvee_{f \in \gamma^{\gamma}}  \exists_{\beta<\gamma}\mathbf{x}_{f|_{\beta +1}} \bigwedge_{\beta<\gamma}\phi_{f|_\beta}}
\end{mathpar}
\\
for each cardinal $\gamma < \kappa$, where $\mathbf{y}_{f}$ is the canonical context of $\phi_{f}$, provided that, for every $f \in \gamma^{\beta+1}$,  $FV(\phi_{f}) = FV(\phi_{f|_{\beta}}) \cup \mathbf{x}_{f}$ and $\mathbf{x}_{f|_{\beta +1}} \cap FV(\phi_{f|_{\beta}})= \emptyset$ for any $\beta<\gamma$, as well as $FV(\phi_{f}) = \bigcup_{\alpha<\beta} FV(\phi_{f|_{\alpha}})$ for limit $\beta$. Note that we assume that there is a fixed well-ordering of $\gamma^{\gamma}$ for each $\gamma<\kappa$.

\end{enumerate}
\end{defs}

In this formulation the double line indicates a bidirectional rule. Note that the axiom (schema) of excluded middle, which is not assumed here, is $\top \vdash_{\mathbf{x}} \phi \vee \neg \phi$. Note also that in full infinitary first-order logic we can dispense with the use of sequents and treat $\phi \vdash_{\mathbf{x}} \psi$ as simply $\forall \mathbf{x}(\phi \to \psi)$. Conversely, any formula $\phi(\mathbf{x})$ can be interpreted as the sequent $\top \vdash_{\mathbf{x}} \phi$, thereby obtaining a translation with Hilbert style systems.

In $\kappa$-first-order logic the following two axioms are provable:

\begin{enumerate}
\item Small distributivity axiom

$$\phi \wedge \bigvee_{i<\gamma} \psi_{i} \vdash_{\mathbf{x}} \bigvee_{i<\gamma}\phi \wedge \psi_{i}$$

for each cardinal $\gamma<\kappa$.

\item Frobenius axiom:

$$\phi \wedge \exists \mathbf{y} \psi \vdash_{\mathbf{x}} \exists \mathbf{y} (\phi \wedge \psi)$$

\noindent where no variable in $\mathbf{y}$ is in the context $\mathbf{x}$.
\end{enumerate}

However, when working in smaller fragments without implication (like the $\kappa$-coherent fragment to be introduced later) they need to be included as they are  not derivable (for the small distributivity axiom this can be seen for example in the propositional case by considering non-distributive lattices satisfying all other axioms, and a similar idea works for Frobenius axiom considering categorical semantics).

The transfinite transitivity rule can be understood as follows. Consider $\gamma^{\leq \gamma}$, the $\gamma$-branching tree of height $\gamma$, i.e., the poset of functions $f: \beta \to \gamma$ for $\beta \leq \gamma$ with the order given by inclusion. Suppose there is an assignment of formulas $\phi_f$ to each node $f$ of $\gamma^{\leq \gamma}$. Then the rule expresses that if the assignment is done in a way that the formula assigned to each node entails the join of the formulas assigned to its immediate successors, and if the formula assigned to a node in a limit level is equivalent to the meet of the formulas assigned to its predecessors, then the formula assigned to the root entails the join of the formulas assigned to the nodes in level $\gamma$. 

There is also a version of the deduction theorem that holds here:

\begin{lemma}\label{dt}
 Let $\Sigma$ be a set of sequents and let $\sigma$ be a sentence. If the theory $\Sigma \cup \{\top \vdash \sigma\}$ derives the sequent $\phi \vdash_{\mathbf{x}} \psi$, then the theory $\Sigma$ derives the sequent $\phi \wedge \sigma \vdash_{\mathbf{x}} \psi$.
\end{lemma}

\begin{proof}
 Straightforward induction on the length of the derivation.
\end{proof}

In full first-order logic the transfinite transitivity rule can be replaced by the axiom schema, for each $\gamma < \kappa$:

$$\bigwedge_{f \in \gamma^{\beta}, \beta<\gamma} \forall (\mathbf{y}_{f} \setminus \mathbf{y}_{\emptyset}) \left(\phi_{f} \to \bigvee_{g \in \gamma^{\beta+1}, g|_{\beta}=f} \exists \mathbf{x}_{g} \phi_{g}\right)$$

$$\wedge \bigwedge_{\beta < \gamma, \text{ limit }\beta, f \in \gamma^{\beta}} \forall (\mathbf{y}_{f} \setminus \mathbf{y}_{\emptyset}) \left(\phi_{f} \leftrightarrow \bigwedge_{\alpha<\beta}\phi_{f|_{\alpha}}\right)$$

$$\vdash_{\mathbf{y}_{\emptyset}} \phi_{\emptyset} \to \bigvee_{f \in \gamma^{\gamma}} \exists_{\beta<\gamma}\mathbf{x}_{f|_{\beta +1}} \bigwedge_{\beta<\gamma}\phi_{f|_{\beta}}.$$

There are two particular cases of the transfinite transitivity rule which are of interest:

\begin{enumerate}

\item Distributivity rule:

\begin{mathpar}
\inferrule{\phi_{f} \vdash_{\mathbf{x}} \bigvee_{g \in \gamma^{\beta+1}, g|_{\beta}=f} \phi_{g} \\ \beta<\gamma, f \in \gamma^{\beta} \\\\ \phi_{f} \dashv \vdash_{\mathbf{x}} \bigwedge_{\alpha<\beta}\phi_{f|_{\alpha}} \\ \beta < \gamma, \text{ limit }\beta, f \in \gamma^{\beta}}{\phi_{\emptyset} \vdash_{\mathbf{x}} \bigvee_{f \in \gamma^{\gamma}}\bigwedge_{\beta<\gamma}\phi_{f|_{\beta}}}
\end{mathpar}

for each $\gamma<\kappa$ (we assume that there is a fixed well-ordering of $\gamma^{\gamma}$ for each $\gamma<\kappa$).

\item Dependent choice rule:

\begin{mathpar}
\inferrule{\phi_{\beta} \vdash_{\mathbf{y}_{\beta}} \exists \mathbf{x}_{\beta+1} \phi_{\beta +1} \\ \beta<\gamma \\\\ \phi_{\beta} \dashv \vdash_{\mathbf{y}_{\beta}} \bigwedge_{\alpha<\beta}\phi_{\alpha} \\ \beta \leq \gamma, \text{ limit }\beta}{\phi_{\emptyset} \vdash_{\mathbf{y}_{\emptyset}} \exists_{\beta<\gamma}\mathbf{x}_{\beta +1}\phi_{\gamma}}
\end{mathpar}

for each $\gamma < \kappa$, where $\mathbf{y}_{\beta}$ is the canonical context of $\phi_{\beta}$, provided that, for every $f \in \gamma^{\beta+1}$,  $FV(\phi_{f}) = FV(\phi_{f|_{\beta}}) \cup \mathbf{x}_{f}$ and $\mathbf{x}_{f|_{\beta +1}} \cap FV(\phi_{f|_{\beta}})= \emptyset$ for any $\beta<\gamma$, as well as $FV(\phi_{f}) = \bigcup_{\alpha<\beta} FV(\phi_{f|_{\alpha}})$ for limit $\beta$.

\end{enumerate}

Again, if implication and universal quantification are available in the fragment we are considering, we can instead replace the distributivity and dependent choice rules by axiom schemata expressible with single sequents, for each $\gamma < \kappa$:

$$\bigwedge_{f \in \gamma^{\beta}, \beta<\gamma} \left( \phi_{f} \to \bigvee_{g \in \gamma^{\beta+1}, g|_{\beta}=f} \phi_{g}\right)$$

$$\wedge \bigwedge_{\beta < \gamma, \text{ limit }\beta, f \in \gamma^{\beta}} \left(\phi_{f} \leftrightarrow \bigwedge_{\alpha<\beta}\phi_{f|_{\alpha}}\right) \vdash_{\mathbf{x}} \phi_{\emptyset} \to \bigvee_{f \in \gamma^{\gamma}} \bigwedge_{\beta<\gamma}\phi_{f|_{\beta}}$$
\\
and

$$\bigwedge_{\beta < \gamma} \forall (\mathbf{y}_{\beta} \setminus \mathbf{y}_{\emptyset}) \left(\phi_{\beta} \to \exists \mathbf{x}_{\beta +1} \phi_{\beta +1} \right)$$

$$\wedge \bigwedge_{\beta \leq \gamma, \text{ limit }\beta} \forall (\mathbf{y}_{\beta} \setminus \mathbf{y}_{\emptyset}) \left(\phi_{\beta} \leftrightarrow \bigwedge_{\alpha<\beta}\phi_{\alpha} \right) \vdash_{\mathbf{y}_{\emptyset}} \phi_{\emptyset} \to \exists_{\alpha < \gamma} \mathbf{x}_{\alpha} \phi_{\gamma}.$$

In turn, the rule of dependent choice has as a particular case the rule of choice:

\begin{mathpar}
\inferrule{\phi \vdash_{\mathbf{x}} \bigwedge_{\beta<\gamma} \exists \mathbf{x}_{\beta} \phi_{\beta}}{\phi \vdash_{\mathbf{x}} \exists_{\beta<\gamma}\mathbf{x}_{\beta}\bigwedge_{\beta<\gamma}\phi_{\beta}}
\end{mathpar}

\noindent where the $\mathbf{x}_{\beta}$ are disjoint canonical contexts of the $\phi_{\beta}$. This can be seen by applying dependent choice to the formulas $\psi_{\beta}=\phi \wedge \bigwedge_{\alpha<\beta}\phi_{\alpha+1}$.

\begin{lemma}\label{equivl}
All instances of the classical distributivity axiom:

$$\bigwedge_{i<\gamma}\bigvee_{j<\gamma} \psi_{ij} \vdash_{\mathbf{x}} \bigvee_{f \in \gamma^{\gamma}}\bigwedge_{i<\gamma} \psi_{if(i)}$$
\\
are derivable from those of the axiom: 

$$\bigwedge_{f \in \gamma^{\beta}, \beta<\gamma}\left(\phi_{f} \to \bigvee_{g \in \gamma^{\beta+1}, g|_{\beta}=f} \phi_{g}\right)$$

$$\wedge \bigwedge_{\beta < \gamma, \text{ limit }\beta, f \in \gamma^{\beta}} \left(\phi_{f} \leftrightarrow \bigwedge_{\alpha<\beta}\phi_{f|_{\alpha}}\right) \vdash_{\mathbf{x}} \phi_{\emptyset} \to \bigvee_{f \in \gamma^{\gamma}} \bigwedge_{\beta<\gamma}\phi_{f|_{\beta}}.$$
\\
Moreover, if excluded middle is available and $\kappa$ is inaccessible, the converse holds.

\end{lemma}

\begin{proof}
Assign to the nodes of the tree $\gamma^{< \gamma}$ the following formulas: to the immediate succesors of a node $\phi_{f}$, for $f \in \gamma^{\beta}$, assign the formulas $\psi_{\beta j}$, then set $\phi_{\emptyset}= \top$, and $\phi_{f}=\bigwedge_{\alpha<\beta}\phi_{f|_{\alpha}}$ for $f \in \gamma^{\beta}$ and limit $\beta$. Then we have the axiom $\bigwedge_{i<\gamma}\bigvee_{j<\gamma} \psi_{ij} \vdash_{\mathbf{x}} \bigvee_{g \in \gamma^{\beta+1}, g|_{\beta}=f} \phi_{g}$ for each $f$, from which we can further derive:

$$\bigwedge_{i<\gamma}\bigvee_{j<\gamma} \psi_{ij} \vdash_{\mathbf{x}} \phi_{f} \to \bigvee_{g \in \gamma^{\beta+1}, g|_{\beta}=f} \phi_{g}$$
\\
for each $f$. Thus, applying the distributivity and the cut rule we get:

$$\bigwedge_{i<\gamma}\bigvee_{j<\gamma} \psi_{ij} \vdash_{\mathbf{x}} \bigvee_{f \in \gamma^{\gamma}}\bigwedge_{i<\gamma} \psi_{if(i)}$$
\\
as we wanted.

If excluded middle is available, we have $\top \vdash_{\mathbf{x}} \bigwedge_{f \in \gamma^{\beta}, \beta<\gamma} (\phi_f \vee \neg \phi_f)$. If $I=\{f \in \gamma^{\beta}, \beta<\gamma\}$, the distributivity axiom implies that then $\top \vdash_{\mathbf{x}} \bigvee_{g \in 2^I} \bigwedge_{f \in I} a_g(\phi_f)$, where $a_g(\phi_f)=\phi_f$ if $g(f)=0$ and $a_g(\phi_f)=\neg \phi_f$ if $g(f)=1$. We have therefore:

$$\phi_{\emptyset} \equiv \bigvee_{g \in 2^I} \left(\phi_{\emptyset} \wedge \bigwedge_{f \in I} a_g(\phi_f)\right).$$
\\
If $\phi_{\emptyset} \wedge \bigwedge_{f \in I} a_g(\phi_f)$ is not equivalent to $\bot$, $a_g(\phi_{\emptyset})=\phi_{\emptyset}$, and hence $\phi_{\emptyset}$ is the join:

$$\bigvee_{\phi_{\emptyset} \wedge \bigwedge_{f \in I} a_g(\phi_f) \not\equiv \bot} \bigwedge_{f \in I} a_g(\phi_f).$$

We will prove that each disjunct in this join entails $\bigvee_{f \in I} \bigwedge_{\beta<\gamma}\phi_{f|_{\beta}}$, which is clearly enough. Take one such disjunct, $b=\bigwedge_{f \in I} a_g(\phi_f) \not\equiv \bot$ for a certain $g$. First, notice that since $\top \vdash_{\mathbf{x}} \phi_{f} \to \bigvee_{h \in \gamma^{\beta+1}, h|_{\beta}=f} \phi_{h}$, if $b \vdash_{\mathbf{x}} \phi_{f}$ then $b \wedge \phi_h \not\equiv \bot$ for at least one $h$, in which case it follows that $a_g(\phi_h)=\phi_h$ and hence $b \vdash_{\mathbf{x}} \phi_h$. With this observation, we can inductively define a path $f \in \gamma^{\gamma}$ such that $b \vdash_{\mathbf{x}} \phi_{f|_{\alpha}}$ for each $\alpha<\gamma$. Thus, $b \vdash_{\mathbf{x}} \bigvee_{f \in \gamma^{\gamma}} \bigwedge_{\beta<\gamma}\phi_{f|_{\beta}}$, as we wanted. This finishes the proof.

\end{proof}

\begin{rmk}
 The intuitionistic form of the distributivity law is strictly stronger than the classical version. For example, the interval $[0, 1]$ with the supremum and infimum as join and meet, respectively, satisfies the classical distributivity law for every $\gamma<\kappa$, but the intuitionistic distributivity law fails for $\gamma=\omega_1$. 
\end{rmk}

Note that the distributivity rule implies in particular, due to Lemma \ref{equivl}, the small distributivity axiom, which is a particular case of classical distributivity. 

\begin{lemma}\label{equivl2}
All instances of the classical dependent choice axiom:

$$\bigwedge_{\alpha < \gamma} \forall_{\beta < \alpha}\mathbf{x}_{\beta} \exists \mathbf{x}_{\alpha} \psi_{\alpha} \vdash_{\mathbf{x}} \exists_{\alpha < \gamma} \mathbf{x}_{\alpha} \bigwedge_{\alpha < \gamma} \psi_{\alpha}$$
\\
are derivable from those of the axiom:

$$\bigwedge_{\beta < \gamma} \forall (\mathbf{y}_{\beta} \setminus \mathbf{y}_{\emptyset}) \left(\phi_{\beta} \to \exists \mathbf{x}_{\beta +1} \phi_{\beta +1} \right)$$

$$\wedge \bigwedge_{\beta \leq \gamma, \text{ limit }\beta} (\mathbf{y}_{\beta} \setminus \mathbf{y}_{\emptyset}) \left(\phi_{\beta} \leftrightarrow \bigwedge_{\alpha<\beta}\phi_{\alpha} \right) \vdash_{\mathbf{y}_{\emptyset}} \phi_{\emptyset} \to \exists_{\alpha < \gamma} \mathbf{x}_{\alpha} \phi_{\gamma}.$$
\\
Moreover, if excluded middle is available, the converse holds.

\end{lemma}

\begin{proof}
Define the following formulas: set $\phi_{\emptyset}= (\mathbf{x}=\mathbf{x})$, $\phi_{\beta +1}=\psi_{\beta}$ and $\phi_{\alpha}=\bigwedge_{\beta < \alpha} \phi_{\beta}$ if $\alpha$ is a limit ordinal. Then we have:

$$\bigwedge_{\alpha < \gamma} \forall_{\beta < \alpha}\mathbf{x}_{\beta} \exists \mathbf{x}_{\alpha} \psi_{\alpha} \vdash_{\mathbf{y_{\beta}}} \exists \mathbf{x}_{\beta} \phi_{\beta +1}$$
\\
and in particular,

$$\bigwedge_{\alpha < \gamma} \forall_{\beta < \alpha}\mathbf{x}_{\beta} \exists \mathbf{x}_{\alpha} \psi_{\alpha} \vdash_{\mathbf{y_{\beta}}} \phi_{\beta} \to \exists \mathbf{x}_{\beta} \phi_{\beta +1}.$$
\\
Taking universal quantification and conjunctions on the right side, together with the sentence $\bigwedge_{\beta \leq \gamma, \text{ limit }\beta} \forall (\mathbf{y}_{\beta} \setminus \mathbf{y}_{\emptyset}) \left(\phi_{\beta} \leftrightarrow \bigwedge_{\alpha<\beta}\phi_{\alpha} \right)$, we can then use dependent choice and the cut rule to get:

$$\bigwedge_{\alpha < \gamma} \forall_{\beta < \alpha}\mathbf{x}_{\beta} \exists \mathbf{x}_{\alpha} \psi_{\alpha} \vdash_{\mathbf{x}} (\mathbf{x}=\mathbf{x}) \to \exists_{\alpha < \gamma} \mathbf{x}_{\alpha} \bigwedge_{\alpha < \gamma} \psi_{\alpha},$$
\\
as desired.

If excluded middle is available, notice that we have:

$$\forall (\mathbf{y}_{\beta} \setminus \mathbf{y}_{\emptyset}) \left(\phi_{\beta} \to \exists \mathbf{x}_{\beta +1} \phi_{\beta +1} \right) \vdash_{\mathbf{y}_{\emptyset}} \forall (\mathbf{y}_{\beta} \setminus \mathbf{y}_{\emptyset}) \exists \mathbf{x}_{\beta +1} \left(\phi_{\beta} \to \phi_{\beta +1} \right)$$
\\
for $\beta < \gamma$ by the independence of premise principle, classically provable. Taking conjunctions and applying classical dependent choice and the cut rule, we get:

$$\bigwedge_{\beta < \gamma} \forall (\mathbf{y}_{\beta} \setminus \mathbf{y}_{\emptyset}) \left(\phi_{\beta} \to \exists \mathbf{x}_{\beta +1} \phi_{\beta +1} \right) \vdash_{\mathbf{y}_{\emptyset}} \exists_{\alpha < \gamma} \mathbf{x}_{\alpha} \bigwedge_{\beta<\gamma} \left(\phi_{\beta} \to \phi_{\beta +1} \right)$$
\\
Since we also have 

$$\bigwedge_{\beta \leq \gamma, \text{ limit }\beta} \forall (\mathbf{y}_{\beta} \setminus \mathbf{y}_{\emptyset}) \left(\phi_{\beta} \leftrightarrow \bigwedge_{\alpha<\beta}\phi_{\alpha} \right) \wedge \exists_{\alpha < \gamma} \mathbf{x}_{\alpha} \bigwedge_{\beta<\gamma} \left(\phi_{\beta} \to \phi_{\beta +1} \right) \vdash_{\mathbf{y}_{\emptyset}} \phi_{\emptyset} \to \exists_{\alpha < \gamma} \mathbf{x}_{\alpha} \bigwedge_{\beta<\gamma} \phi_{\beta}$$
\\
the result follows.
\end{proof}

We have seen that the transfinite transitivity rule implies both classical distributivity and classical dependent choice. We will see now that if excluded middle is available, the converse holds:

\begin{thm}\label{equiv}
Assuming excluded middle, if $\kappa$ is inaccessible, all the instances of the transfinite transitivity rule:

\begin{mathpar}
\inferrule{\phi_{f} \vdash_{\mathbf{y}_{f}} \bigvee_{g \in \gamma^{\beta+1}, g|_{\beta}=f} \exists \mathbf{x}_{g} \phi_{g} \\ \beta<\gamma, f \in \gamma^{\beta} \\\\ \phi_{f} \dashv \vdash_{\mathbf{y}_{f}} \bigwedge_{\alpha<\beta}\phi_{f|_{\alpha}} \\ \beta < \gamma, \text{ limit }\beta, f \in \gamma^{\beta}}{\phi_{\emptyset} \vdash_{\mathbf{y}_{\emptyset}} \bigvee_{f \in \gamma^{\gamma}}  \exists_{\beta<\gamma}\mathbf{x}_{f|_{\beta +1}} \bigwedge_{\beta<\gamma}\phi_{f|_\beta}}
\end{mathpar}
\\ 
are derivable from instances of the classical distributivity axiom:

$$\bigwedge_{i<\gamma}\bigvee_{j<\gamma} \psi_{ij} \vdash_{\mathbf{x}} \bigvee_{f \in \gamma^{\gamma}}\bigwedge_{i<\gamma} \psi_{if(i)}$$
\\and the classical dependent choice axiom:

$$\bigwedge_{\alpha < \gamma} \forall_{\beta < \alpha}\mathbf{x}_{\beta} \exists \mathbf{x}_{\alpha} \psi_{\alpha} \vdash_{\mathbf{x}} \exists_{\alpha < \gamma} \mathbf{x}_{\alpha} \bigwedge_{\alpha < \gamma} \psi_{\alpha}.$$
\\
In particular, the addition to the axiom of excluded middle to the system of $\kappa$-first-order logic results in a system equivalent to Karp's.
\end{thm}

\begin{proof}
If $\phi_{f} \vdash_{\mathbf{y}_{f}} \bigvee_{g \in \gamma^{\beta+1}, g|_{\beta}=f} \exists \mathbf{x}_{g} \phi_{g}$ for each $\beta<\gamma$, $f \in \gamma^{\beta}$, then we can derive, using the independence of premise principle, the sequent:

$$\top \vdash_{\mathbf{y}_{\emptyset}} \forall_{\alpha < \beta}(\cup_{f \in \gamma^{\beta}}\mathbf{x}_{f|_{\alpha+1}}) \bigwedge_{f \in \gamma^{\beta}} \exists_{g \in \gamma^{\beta+1}, g|_{\beta}=f} \mathbf{x}_{g} \left(\phi_f \to \bigvee_{g \in \gamma^{\beta+1}, g|_{\beta}=f} \phi_{g}\right).$$
\\
Using the rule of choice, derivable from that of dependent choice (which is in turn classically derivable from classical dependent choice by Lemma \ref{equivl2}), we get:

$$\top \vdash_{\mathbf{y}_{\emptyset}} \forall_{\alpha < \beta}(\cup_{f \in \gamma^{\beta}}\mathbf{x}_{f|_{\alpha+1}}) \exists (\cup_{g \in \gamma^{\beta+1}}\mathbf{x}_{g}) \bigwedge_{f \in \gamma^{\beta}} \left(\phi_f \to \bigvee_{g \in \gamma^{\beta+1}, g|_{\beta}=f} \phi_{g}\right).$$
\\
Since we have one such sequent for each $\beta<\gamma$, by classical dependent choice we can thus infer:

$$\top \vdash_{\mathbf{y}_{\emptyset}} \exists_{\alpha < \gamma} (\cup_{f \in \gamma^{\gamma}}\mathbf{x}_{f|_{\alpha+1}}) \bigwedge_{\beta<\gamma, f \in \gamma^{\beta}} \left(\phi_f \to \bigvee_{g \in \gamma^{\beta+1}, g|_{\beta}=f} \phi_{g}\right). \qquad (1)$$
\\
On the other hand, by the distributivity property (provable by Lemma \ref{equivl} from the classical distributivity property, since excluded middle is available), we have:

$$\bigwedge_{\beta<\gamma, f \in \gamma^{\beta}} \left(\phi_f \to \bigvee_{g \in \gamma^{\beta+1}, g|_{\beta}=f} \phi_{g}\right) \wedge \bigwedge_{\beta < \gamma, \text{ limit }\beta, f \in \gamma^{\beta}} \left(\phi_{f} \leftrightarrow \bigwedge_{\alpha<\beta}\phi_{f|_{\alpha}}\right)$$

$$\vdash_{\mathbf{y}_{\emptyset} \cup \bigcup_{\alpha < \gamma, f \in \gamma^{\gamma}} \mathbf{x}_{f|_{\alpha +1}}} \phi_{\emptyset} \to \bigvee_{f \in \gamma^{\gamma}} \bigwedge_{\beta<\gamma}\phi_{f|_\beta}.$$
\\
Therefore, the sequent $(1)$ together with the premises $\top \vdash_{\mathbf{y}_{\emptyset} \cup \bigcup_{\alpha < \gamma, f \in \gamma^{\gamma}} \mathbf{x}_{f|_{\alpha +1}}} \phi_{f} \leftrightarrow \bigwedge_{\alpha<\beta}\phi_{f|_{\alpha}}$ for $\beta < \gamma$, $\text{ limit }\beta$, $f \in \gamma^{\beta}$, entail:

$$\top \vdash_{\mathbf{y}_{\emptyset}} \phi_{\emptyset} \to \exists_{\alpha < \gamma} (\cup_{f \in \gamma^{\gamma}}\mathbf{x}_{f|_{\alpha+1}}) \bigvee_{f \in \gamma^{\gamma}} \bigwedge_{\beta<\gamma}\phi_{f|_\beta}$$
\\
which in turn entails:

$$\top \vdash_{\mathbf{y}_{\emptyset}} \phi_{\emptyset} \to \bigvee_{f \in \gamma^{\gamma}} \exists_{\beta<\gamma}\mathbf{x}_{f|_{\beta +1}} \bigwedge_{\beta<\gamma}\phi_{f|_\beta}.$$
\\
This concludes the proof.
\end{proof}

\section{Infinitary first-order categorical logic}

\subsection{$\kappa$-coherent categories}

We begin now to study the categorical counterpart of the infinitary systems. We start with the following:

\begin{defs}
The $\kappa$-coherent fragment of $\kappa$-first-order logic is the fragment of those sequents where formulas are $\kappa$-coherent, i.e., only use $\bigwedge$, $\bigvee$, $\exists$, where the axioms and rules of $\kappa$-first-order logic are restricted to instantiations on $\kappa$-coherent formulas only, and where we add the corresponding instances of the small distributivity and the Frobenius axioms. 
\end{defs}

The $\kappa$-coherent fragment of first-order logic, which is an extension of the usual finitary coherent fragment, has a corresponding category which we are now going to define. Following \cite{makkai}, consider a $\kappa$-chain in a category $\mathcal{C}$ with $\kappa$-limits, i.e., a diagram $\Gamma: \gamma^{op} \to \mathcal{C}$ specified by morphisms $(h_{\beta, \alpha}: C_{\beta} \to C_{\alpha})_{\alpha \leq \beta<\gamma}$ such that the restriction $\Gamma|_{\beta}$ is a limit diagram for every limit ordinal $\beta$. We say that the morphisms $h_{\beta, \alpha}$ compose transfinitely, and take the limit projection $f_{\beta, 0}$ to be the transfinite composite of $h_{\alpha+1, \alpha}$ for $\alpha<\beta$.

Given a cardinal $\gamma<\kappa$, consider the tree $T=\gamma^{<\gamma}$. We will consider diagrams $F: T^{op} \to \mathcal{C}$, which determine, for each node $f$, a family of arrows in $\mathcal{C}$, $\{h_{g, f}: C_{g} \to C_{f} | f \in \gamma^{\beta}, g \in \gamma^{\beta+1}, g|_{\beta}=f\}$. A $\kappa$-family of morphisms with the same codomain is said to be \emph{jointly covering} if the union of the images of the morphisms is the whole codomain. We say that a diagram $F: T^{op} \to \mathcal{C}$ is \emph{proper} if the $\{h_{g, f}: f \in T\}$ are jointly covering and, for limit $\beta$, $h_{f, \emptyset}$ is the transfinite composition of the $h_{f|_{\alpha+1}, f|_{\alpha}}$ for $\alpha+1<\beta$. Given a proper diagram, we say that the families $\{h_{g, f}: f \in T\}$ compose transfinitely, and refer to the projections $\{h_{g, \emptyset} | g \in \gamma^{\gamma}\}$ as the transfinite composites of these families. If in a proper diagram the transfinite composites of the $\kappa$-families of morphisms form itself a jointly covering family, we will say that the diagram is completely proper.

\begin{defs}
 A $\kappa$-coherent category is a $\kappa$-complete coherent category with $\kappa$-complete subobject lattices where unions of cardinality less than $\kappa$ are stable under pullback, and where every proper diagram is completely proper, i.e., the transfinite composites of jointly covering $\kappa$-families of morphisms form a jointly covering family.
\end{defs}

The latter property, which is the categorical analogue of the transfinite transitivity rule, can be considered as an exactness property of $\Sets$, generalizing the property in \cite{makkai} where the families consisted of single morphisms. The transfinite transitivity property expresses that transfinite compositions of covering families (in the Grothendieck topology given by the jointly covering families of less than $\kappa$-morphisms) are again covering families; whence its name. It is easy to see that the property holds in \Sets, and in fact in every presheaf category.

$\kappa$-coherent categories have an internal logic, in a signature containing one sort for each object, no relation symbols and one unary function symbol for each arrow, and axiomatized by the following sequents:

$$\top \vdash_x \it Id_X(x)=x$$
\\
for all objects $X$ (here $x$ is a variable of sort $X$);

$$\top \vdash_x f(x)=h(g(x))$$
\\
for all triples of arrows such that $f=h \circ g$ (here $x$ is a variable whose sort is the domain of $f$);

$$\top \vdash_y \exists x f(x)=y$$
\\
for all covers $f$ (here $x$ is a variable whose sort is the domain of $f$);

$$\top \vdash_x \bigvee_{i<\gamma}\exists y_i m_i(y_i)=x$$
\\
whenever the sort $A$ of $x$ is the union of $\gamma$ subobjects $m_i: A_i \rightarrowtail A$ (here $y_i$ is a variable of sort $A_i$);

$$\bigwedge_{i:I \to J}\overline{i}(x_I)=x_J \vdash_{\{x_I: I \in \mathbf{I}\}} \exists x \bigwedge_{I \in \mathbf{I}} \pi_I(x)=x_I$$

$$\bigwedge_{I \in \mathbf{I}}\pi_I(x)=\pi_I(y) \vdash_{x, y} x=y$$
\\
whenever there is a $\kappa$-small diagram $\Phi: \mathbf{I} \to \mathcal{C}$, $(\{C_I\}_{I \in \mathbf{I}}, \{\overline{i}: C_I \to C_J)\}_{i: I \to J})$ and a limit cone $\pi: \Delta_C \Rightarrow \Phi$, $(\pi_I: C \to C_I)_{I \in \mathbf{I}}$. Here $x_I$ is a variable of type $C_I$, and $x, y$ are variables of type $C$.

Functors preserving this logic, i.e., $\kappa$-coherent functors, are just coherent functors which preserve $\kappa$-limits and $\kappa$-unions of subobjects, and they can be easily seen to correspond to structures of the internal theory in a given $\kappa$-coherent category, where we use a straightforward generalization of categorical semantics, to be explained in the next section.

\subsection{Categorical semantics and Kripke-Joyal forcing}

Categorical model theory techniques explore the study of models in arbitrary categories besides the usual category of sets. Unlike classical model theory, the logics one uses for this purpose formulate theories in terms of sequents; the type of theory studied depends on the type of formula one encounters in these sequents. The theories of the fragments mentioned so far all correspond to specific types of categories. We have the $\kappa$-regular categories, which are categories with $\kappa$-limits, regular epimorphism-monomorphism factorizations stable under pullback and where the transfinite composition of epimorphisms is an epimorphism. The $\kappa$-coherent categories have, in addition to this, stable $\kappa$-unions of subobjects and satisfy the property that the transfinite composition of jointly covering families is jointly covering. Finally, the $\kappa$-Heyting categories have, in addition, right adjoint for pullback functors between subobject lattices, which makes interpreting universal quantification possible.

There is a categorical semantics that one can associate with each type of category and theory, which is usually defined according to some inductive clauses. Following \cite{johnstone}, D1.2, given a category $\mathcal{C}$, for each signature $\Sigma$ of a first order language we can associate the so called $\Sigma$-structure within $\mathcal{C}$ in a way that generalizes the $\Sets$-valued interpretations to all $\kappa$-Heyting categories:

\begin{defs}
 A $\Sigma$-structure in $\mathcal{C}$ consists of the following data:
 \begin{enumerate}
  \item for each sort $A$ of variables in $\Sigma$ there is a corresponding object $M(A)$;
  \item for each $\gamma$-ary function symbol $f$ there is a morphism $M(f): M(A_1, ..., A_{\alpha}, ...)=\Pi_{i<\gamma} M(A_i) \rightarrow M(B)$;
  \item for each $\gamma$-ary relation symbol $R$ there is a subobject $M(R) \rightarrowtail M(A_1, ..., A_{\alpha}, ...)$, where $A_i$ are the sorts corresponding to the individual variables corresponding to $R$ (which will specify, by definition, the type of $R$).
 \end{enumerate}

\end{defs}

The $\Sigma$-structure will serve as a setup for interpreting all formulas of the language considered. Due to the need of distinguishing the context in which the free variables of the formula occur, for the purpose of a correct interpretation, we shall adopt the notation $(\mathbf{x}, \phi)$ to represent a term/formula $\phi$ whose free variables occur within $\mathbf{x}=x_1, ...,x_{\alpha}, ...$. We now define the interpretation of such formulas by induction on their complexity:

\begin{defs} Given a term in context $(\mathbf{x}, s)$ of a $\kappa$-first order theory, its interpretation $[\![\mathbf{x}, s]\!]$ within the $\kappa$-Heyting category $\mathcal{C}$ is a morphism of $\mathcal{C}$ defined in the following way:
\begin{enumerate}

\item If $s$ is a variable, it is necessarily some $x_i$, and then the corresponding morphism is $[\![\mathbf{x}, x_i]\!]=\pi_i: M(A_0, ..., A_{\alpha}, ...) \rightarrow M(A_i)$, the $i$-th product projection.
\item If $s$ is a term $f(t_0, ..., t_{\alpha}, ...)$, where each term $t_{\alpha}$ is of type $C_{\alpha}$, its interpretation is the composite:

\begin{displaymath}
\xymatrix{
M(A_0, ..., A_{\alpha}, ...) \ar@{->}[rrr]^{([\![\mathbf{x}, t_0]\!], ..., [\![\mathbf{x}, t_{\alpha}]\!], ...)} & & & M(C_0, ..., C_{\alpha}, ...) \ar@{->}[rr]^{M(f)} & & M(B)\\
}
\end{displaymath}

\end{enumerate}

The interpretation in $\mathcal{C}$ of the formula in context $(\mathbf{x}, \phi)$, where $\mathbf{x}=x_0...x_{\alpha} ...$ and $x_i$ is a variable of sort $A_i$, is defined as a subobject $[\![\mathbf{x}, \phi]\!] \rightarrowtail M(A_0, ..., A_{\alpha}, ...)$ in the following way:
\begin{enumerate}

\item If $\phi$ is the formula $R(t_0, ..., t_{\alpha}, ...)$, where $R$ is a $\gamma$-ary relation symbol of type $B_0, ..., B_{\alpha}, ...$, then $[\![\mathbf{x}, \phi]\!]$ is given by the pullback:

\begin{displaymath}
\xymatrix{
[\![\mathbf{x}, \phi]\!] \ar@{->}[rrr] \ar@{ >->}[dd] & & & M(R) \ar@{ >->}[dd]\\
 & & & \\
M(A_0, ..., A_{\alpha}, ...) \ar@{->}[rrr]^{([\![\mathbf{x}, t_1]\!], ..., [\![\mathbf{x}, t_{\alpha}]\!], ...)} & & & M(B_0, ..., B_{\alpha}, ...)\\
}
\end{displaymath}

\item If $\phi$ is the formula $s=t$ where $s, t$ are terms of sort $B$, then $[\![\mathbf{x}, \phi]\!]$ is the equalizer of the arrows:

\begin{displaymath}
\xymatrix{
M(A_0, ..., A_{\alpha}, ...) \ar@/^{1pc}/[rr]^{[\![\mathbf{x}, s]\!]} \ar@/_{1pc}/[rr]_{[\![\mathbf{x}, t]\!]} & & M(B)\\
}
\end{displaymath}
\\
Equivalently, $[\![\mathbf{x}, \phi]\!]$ is the pullback of the diagonal $M(B) \rightarrowtail M(B) \times M(B)$ along the morphism $([\![\mathbf{x}, s]\!], [\![\mathbf{x}, t]\!])$.
\item If $\phi$ is the formula $\bigvee_{i<\gamma}\psi_i$, then $[\![\mathbf{x}, \phi]\!]$ is the union $\bigvee_{i<\gamma}[\![\mathbf{x}, \psi_i]\!]$ in\\ $\mathcal{S}ub(M(A_0, ..., A_{\alpha}, ...))$. If $\phi$ is the formula $\bigwedge_{i<\gamma}\psi_i$, then $[\![\mathbf{x}, \phi]\!]$ is the intersection $\bigwedge_{i<\gamma}[\![\mathbf{x}, \psi_i]\!]$ in $\mathcal{S}ub(M(A_0, ..., A_{\alpha}, ...))$. Similarly, if $\phi$ is the formula $\neg \psi$, the corresponding subobject is $\neg[\![\mathbf{x}, \psi]\!]$.

\item If $\phi$ is the formula $(\exists y)\psi$, then $[\![\mathbf{x}, \phi]\!]$ is the image of the composite:

\begin{displaymath}
\xymatrix{
[\![\mathbf{x}y, \psi]\!] \ar@{ >->}[r] & M(A_0, ..., A_{\alpha}, ..., B) \ar@{->}[r]^{\pi} & M(A_0, ..., A_{\alpha}, ...)\\
}
\end{displaymath}
\\
where $\pi$ is the projection to the first $\gamma$ coordinates. Equivalently, this amounts to applying the left adjoint $\exists_{\pi}$ to the pullback functor\\ $\pi^{-1}: \mathcal{S}ub(M(A_0, ..., A_{\alpha}, ..., B)) \to \mathcal{S}ub(M(A_0, ..., A_{\alpha}, ...))$.
\item If $\phi$ is the formula $(\forall y)\psi$, then $[\![\mathbf{x}, \phi]\!]$ can be obtained by applying to $[\![\mathbf{x}y, \psi]\!]$ the right adjoint $\forall_{\pi}$ to the pullback functor\\ $\pi^{-1}: \mathcal{S}ub(M(A_0, ..., A_{\alpha}, ..., B)) \to \mathcal{S}ub(M(A_0, ..., A_{\alpha}, ...))$, where $\pi$ is the projection to the first $\gamma$  coordinates. Implication can be seen as a particular case of this right adjoint, by considering in $\mathcal{S}ub(M(A_0, ..., A_{\alpha}, ...))$ the pullback functor $\phi \wedge -: \mathcal{S}ub(M(A_0, ..., A_{\alpha}, ...)) \to \mathcal{S}ub(M(A_0, ..., A_{\alpha}, ...))$.
\end{enumerate}
\end{defs}

Functors between the appropriate categories preserving the corresponding structure correpond to models in the codomain category of the internal theory of the domain category. Such functors are called conservative if they reflect isomorphisms, and hence they reflect also the validity of formulas in the corresponding models.

One then has:

\begin{lemma}\label{soundness}
$\kappa$-coherent logic is sound with respect to models in $\kappa$-coherent categories.
\end{lemma}

\begin{proof}
 This is straightforward for all axioms and rules, except for the rule of transfinite transitivity. But here the proof is the natural generalization of that of the soundness of dependent choice, presented in \cite{makkai} for $\kappa$-regular logic. Let $S_{\mathbf{y}_{f}}$ be the product of the sorts assigned to the variables in $\mathbf{y}_{f}$ in the structure within a $\kappa$-coherent category, and assume that the premises of the transfinite transitivity rule hold there. We must show that the conclusion holds. We can also assume, without loss of generality, that:
 
 $$\phi_{g} \vdash_{\mathbf{y}_{g}} \phi_{f}$$
 \\
 for each $g \in \gamma^{\beta+1}, g|_{\beta}=f$; otherwise we can take, for each $f \in \gamma^{\beta}$:
 
 $$\psi_{f}=\bigwedge_{\alpha \leq \beta} \phi_{f|_{\alpha}}$$
 \\
 which, using the small distributivity law as well as Frobenius axiom can be seen to satisfy the premises of the rule as well, and both this form of distributivity and Frobenius axiom hold in any $\kappa$-coherent category because $\kappa$-unions and covers are stable under pullback.
 
 Let $m_{\alpha}: C_{\alpha} \to S_{\mathbf{y}_{f}}$ be the monomorphism representing the subobject $[\mathbf{y}_{f}, \phi_{f}]$. The assumption we have provides arrows:
 
 $$h_{g, f}: C_{g} \to C_{f}$$
 \\
 for $g \in \gamma^{\beta+1}$, $g|_{\beta}=f$, and by interpreting the premises of the rule it follows that the arrows:
 
 $$\{h_{g, f} | g \in \gamma^{\beta+1}, g|_{\alpha}=f\}$$
 \\
 form a jointly covering family. For a fixed $f \in \gamma^{\gamma}$ and limit $\beta$, the limit of the diagram formed by the $C_{f|_{\alpha}}$ for $\alpha<\beta$ is given by the intersection in the subobject lattice of $S_{\mathbf{y}_{f|_{\beta}}}$ of the pullbacks of each $m_{\alpha}$ along the projections $\pi_{f|_{\beta}, f|_{\alpha}}: S_{\mathbf{y}_{f|_{\beta}}} \to S_{\mathbf{y}_{f|_{\alpha}}}$. This intersection is in turn given by the subobject:
 
 $$C_{f|_{\beta}}=\bigwedge_{\alpha<\beta}\phi_{f|_{\alpha}} \to S_{\mathbf{y}_{f|_{\beta}}}$$
 \\
 
 By the property of the $\kappa$-coherent category, the arrows $C_{f|_{\beta}} \to C_{\emptyset}$ for $f \in \gamma^{\gamma}$ form a jointly covering family whenever $\beta$ is a limit ordinal, and the interpretation of the conclusion of the rule is precisely this statement for the case $\beta=\gamma$. This proves the soundness of the rule.
\end{proof}

One particular case of categorical semantics is given by the so called Kripke-Joyal semantics in a topos. A sheaf topos is in particular a Heyting category, and so it will be $\kappa$-Heyting precisely when it satisfies the transfinite transitivity property. The verification of the Heyting properties of a sheaf topos presents no major difficulties, but it is instructive to point how the connectives and quantifiers are interpreted. Following \cite{maclane-moerdijk}, section III.8, given a sheaf topos $\mathcal{S}h(\mathcal{C}, \tau)$ and some subsheaves $\{A, B, A_i: i \in I\}$ in the lattice of subobjects of a sheaf $E$, we have:

\begin{enumerate}
 \item The subsheaf $0$, initial in $\mathcal{S}ub(E)$, is determined by the property $x \in 0(C) \Leftrightarrow \emptyset$ is a cover of $C$ and $x \in E(C)$.
 \item The subsheaf $\bigwedge_{i \in I} A_i$ is determined by the property $\bigwedge_{i \in I} A_i(C)= \bigcap_{i \in I} A_i(C)$.
 \item The subsheaf $\bigvee_{i \in I} A_i$ is determined by the property\footnote{For $e \in E(C)$ and $f: D \to C$, we denote by $e.f \in E(D)$ the element $E(f)(e)$.} $e \in \bigvee_{i \in I} A_i(C) \Leftrightarrow \{f: D \to C/e.f \in \bigcup_{i \in I}A_i(D)\}$ is a cover of $C$.
 \item The subsheaf $A \rightarrow B$ is determined by the property $e \in (A \rightarrow B)(C) \Leftrightarrow$ for all $f: D \to C$, $e \in A(D)$ implies $e \in B(D)$.
 \item Given a morphism $\phi: E \to F$ inducing the pullback functor between subsheaves $\phi^{-1}: \mathcal{S}ub(F) \to \mathcal{S}ub(E)$, the action of its left adjoint $\exists_{\phi}$ is determined by the property $y \in \exists_{\phi}(A)(C) \Leftrightarrow \{f: D \to C/\exists a \in A(D): \phi_D(a)=y.f\}$ is a cover of $C$.
 \item Given a morphism $\phi: E \to F$ inducing the pullback functor between subsheaves $\phi^{-1}: \mathcal{S}ub(F) \to \mathcal{S}ub(E)$, the action of its right adjoint $\forall_{\phi}$ is determined by the property $y \in \forall_{\phi}(A)(C) \Leftrightarrow$ for all $f: D \to C$, $\phi_D^{-1}(y.f) \subseteq A(D)$.
\end{enumerate}

If we have a model of a theory in a sheaf topos and $C$ is now any object from $\mathcal{C}$, the forcing relation $C \Vdash \phi(\boldsymbol\alpha)$ for $\boldsymbol\alpha: [-, C] \to [\![\mathbf{x}, \top]\!]$ holds by definition if $\boldsymbol\alpha$ factors through the subobject $[\![\mathbf{x}, \phi(\mathbf{x})]\!] \rightarrowtail [\![\mathbf{x}, \top]\!]$. From the previous clauses, one can see that the following relations hold:

\begin{enumerate}
 \item $C \Vdash \bot \Leftrightarrow$ $\emptyset$ is a cover of $C$. 
 \item $C \Vdash \bigwedge_{i \in I} \phi_i(\boldsymbol\alpha) \Leftrightarrow C \Vdash \phi_i(\boldsymbol\alpha)$ for every $i \in I$.
 \item $C \Vdash \bigvee_{i \in I} \phi_i(\boldsymbol\alpha) \Leftrightarrow $ there is a cover $\{f_j: C_j \to C\}_{j \in J}$ such that for every $j \in J$, $C_j \Vdash \phi_{i_j}(\boldsymbol\alpha f_j)$ for some $i_j \in I$ (we say that the covering family witnesses the forcing clause).
 \item $C \Vdash \phi(\boldsymbol\alpha) \to \psi(\boldsymbol\alpha) \Leftrightarrow $ for all $f: D \to C$, $D \Vdash \phi(\boldsymbol\alpha f)$ implies $D \Vdash \psi(\boldsymbol\alpha f)$.
 \item $C \Vdash \exists \mathbf{x}\phi(\mathbf{x}, \boldsymbol\alpha) \Leftrightarrow $ there is a cover $\{f_j: C_j \to C\}_{j \in J}$ such that for every $j \in J$, $C_j \Vdash \phi(\mathbf{c_j}, \boldsymbol\alpha f_j)$ for some $\mathbf{c_j}: C_{j} \to [\![\mathbf{x}, \top]\!]$ (we say that the covering family witnesses the forcing clause).
 \item $C \Vdash \forall \mathbf{x}\phi(\mathbf{x}, \boldsymbol\alpha) \Leftrightarrow $ for all $f: D \to C$ and $\mathbf{c}: D \to [\![\mathbf{x}, \top]\!]$ we have $D \Vdash \phi(\mathbf{c}, \boldsymbol\alpha f)$.
\end{enumerate}

(These last two clauses are obtained from the case in which the morphism $\phi$ is the projection between the products of corresponding sorts). There are two important properties of this forcing notion. First, it is clear that if $C \Vdash \phi(\boldsymbol\alpha)$ and $f: D \to C$, then $D \Vdash \phi(\boldsymbol\alpha f)$. The second property states that if $\{f_i: C_i \to C\}_{i \in I}$ is a cover of $C$ and $C_i \Vdash \phi(\boldsymbol\alpha f_i)$ for all $i \in I$, then also $C \Vdash \phi(\boldsymbol\alpha)$. This is simply a reformulation of the glueing condition for the subsheaf $[\![\mathbf{x}, \phi(\mathbf{x})]\!] \rightarrowtail [\![\mathbf{x}, \top]\!]$.

The Kripke-Joyal forcing and its properties will become useful in our proof of completeness of infinitary logics.

\subsection{Syntactic categories}

The development of syntactic categories for infinite quantifier logics follows precisely the same pattern as the finitary case, except that instead of finite contexts for the objects of the syntactic category of a theory over such logic, we allow arbitrary sets of variables of cardinality less than $\kappa$, following, e.g., \cite{makkai}. 

Given a $\kappa$-coherent theory $\theory$, we explain how to define, following \cite{johnstone}, D1.4, its syntatic category $\mathcal{C}_{\theory}$ and a categorical model $M_{\theory}$ inside it, in such a way that a formula in $\theory$ will be provable if and only if its interpretation in $\mathcal{C}_{\theory}$ is satisfied by the model $M_{\theory}$. Formulas shall be considered in suitable contexts, which are (possibly empty) subsets of variables of cardinality less than $\kappa$ containing the free variables of the formula. We will say that two formulas in context $(\mathbf{x}, \phi), (\mathbf{y}, \psi)$ are $\alpha$-equivalent if the second has been obtained from the first after renaming the bound variables of $\phi$ and the variables in the context (some of them appearing as free variables in $\phi$). We take the objects of $\mathcal{C}_{\theory}$ to be the $\alpha$-equivalence classes of formulas $(\mathbf{x}, \phi)$. To describe the morphisms, consider two objects $[\mathbf{x}, \phi], [\mathbf{y}, \psi]$, and assume, without loss of generality, that their set of variables $\mathbf{x}, \mathbf{y}$ are disjoint. Consider now a formula $\theta$ that satisfies the following conditions:\\
\\
a) Its free variables are amongst $\mathbf{x}\mathbf{y}$.\\
b) The following sequents are provable in $\theory$:\\

$$ \theta(\mathbf{x}, \mathbf{y}) \vdash_{\mathbf{x}\mathbf{y}} \phi(\mathbf{x}) \wedge \psi(\mathbf{y})$$
$$ \phi(\mathbf{x}) \vdash_{\mathbf{x}} \exists \mathbf{y} (\theta(\mathbf{x},\mathbf{y}))$$
$$ \theta(\mathbf{x}, \mathbf{y}) \wedge \theta(\mathbf{x}, \mathbf{z}/\mathbf{y}) \vdash_{\mathbf{x}\mathbf{y}\mathbf{z}} (\mathbf{y}=\mathbf{z})$$
\\
Define now the morphisms between $[\mathbf{x}, \phi]$ and $[\mathbf{y}, \psi]$ to be the provable-equivalence class of all those formulas of $\theory$ that satisfy conditions a) and b) above. 

The idea behind this definition is to allow only those morphisms that are exactly needed for our purposes. More precisely, the first formula in condition b) restricts the interpretation $[\![\theta(\mathbf{x}, \mathbf{y})]\!]$ in any model to be a subobject of $[\![\phi(\mathbf{x}) \wedge \psi(\mathbf{y})]\!]$, while the last two formulas imply, if the category has finite limits, that it will be the graph of a morphism from $[\![\phi(\mathbf{x})]\!]$ to $[\![\psi(\mathbf{y})]\!]$. Because of the particular construction of the category $\mathcal{C}_{\theory}$, this says exactly that the class $[\mathbf{x} \mathbf{y}, \theta(\mathbf{x}, \mathbf{y})]$ is a morphism from $[\mathbf{x}, \phi(\mathbf{x})]$ to $[\mathbf{y}, \psi(\mathbf{y})]$.

The composite of two morphisms:

\begin{displaymath}
\xymatrix{
[\mathbf{x}, \phi] \ar@{->}[r]^{[\mathbf{x}\mathbf{y},\theta]} & [\mathbf{y}, \psi] \ar@{->}[r]^{[\mathbf{y}\mathbf{z},\delta]} & [\mathbf{z}, \eta]\\
}
\end{displaymath}
\\
is defined to be the class $[\mathbf{x}\mathbf{z}, \exists \mathbf{y} (\theta \wedge \delta)]$. It can be verified that this definition does not depend on the choice of representatives $\theta, \delta$ and that this morphism so defined satisfies conditions a) and b) above. It can also be verified that composition of morphisms is associative. Finally, the identity morphism on an object $[\mathbf{x}, \phi]$ can be defined to be arrow:

\begin{displaymath}
\xymatrix{
[\mathbf{x}, \phi] \ar@{->}[rrr]^{[\mathbf{x} \mathbf{y}, \phi(\mathbf{x}) \wedge (\mathbf{x}=\mathbf{y})]} & & & [\mathbf{y}, \phi(\mathbf{y}/\mathbf{x})]\\
}
\end{displaymath}
\\
Again, it is easily checked that this morphism satisfies condition a) and b) and that it is the unity for composition. Also, note that these definitions do not depend on the choices of representatives in each class. This makes $\mathcal{C}_{\theory}$ a small category.

Our goal is to relate syntactical provability in $\theory$ with semantic validity in the categorical model $M_{\theory}$ to be defined. One aspect of this relation is given by the following lemma, which highlights the syntactical properties of $\mathcal{C}_{\theory}$:

\begin{lemma}\label{lemmap1}
1) A morphism $[\mathbf{x}\mathbf{y}, \theta]: [\mathbf{x}, \phi] \rightarrow [\mathbf{y}, \psi]$ is an isomorphism if and only if $[\mathbf{y}\mathbf{x}, \theta]: [\mathbf{y}, \psi] \rightarrow [\mathbf{x}, \phi]$ is a valid morphism in $\mathcal{C}_{\theory}$ (i.e., it satisfies conditions a) and b) of the definition of morphism).\\
2) A morphism $[\mathbf{x}\mathbf{y}, \theta]: [\mathbf{x}, \phi] \rightarrow [\mathbf{y}, \psi]$ is a monomorphism if and only if the sequent $ \theta(\mathbf{x}, \mathbf{y}) \wedge \theta(\mathbf{z}, \mathbf{y}) \vdash_{\mathbf{x}\mathbf{y}\mathbf{z}} \mathbf{x}=\mathbf{z})$ is provable in $\theory$.\\
3) Every subobject of $[\mathbf{y}, \phi]$ is isomorphic to one of the form:

\begin{displaymath}
\xymatrix{
[\mathbf{x}, \psi] \ar@{ >->}[rr]^{[\psi \wedge (\mathbf{x}=\mathbf{y})]} & & [\mathbf{y}, \phi]\\
}
\end{displaymath}
\\
where $\psi$ is such that the sequent $\psi(\mathbf{y}) \vdash_{\mathbf{y}} \phi(\mathbf{y})$ is provable in $\theory$. Moreover, any two subobjects $[\mathbf{y}, \psi], [\mathbf{y}, \eta]$ in $\mathcal{S}ub ([\mathbf{y}, \phi])$ satisfy $[\mathbf{y}, \psi] \leq [\mathbf{y}, \eta]$ if and only if the sequent $\psi(\mathbf{y}) \vdash_{\mathbf{y}} \eta(\mathbf{y})$ is provable in $\theory$.\end{lemma}

\begin{proof} To prove 1), suppose $[\mathbf{y}\mathbf{x}, \theta]$ is a valid morphism from $[\mathbf{y}, \psi]$ to $[\mathbf{x}, \phi]$. Then it can be easily checked that $[\mathbf{y}\mathbf{x}, \theta]$ itself is an inverse for $[\mathbf{x}\mathbf{y}, \theta]$. Conversely, if $[\mathbf{x}\mathbf{y}, \theta]: [\mathbf{x}, \phi] \rightarrow [\mathbf{y}, \psi]$ has an inverse $[\mathbf{y}\mathbf{x}, \delta]$ (which is a valid morphism), then it can be verified that $\theta$ and $\delta$ are necessarily provable equivalent in $\theory$, from which the result follows.\\

To prove 2), construct the kernel pair of $[\mathbf{x}\mathbf{y}, \theta]: [\mathbf{x}, \phi] \rightarrow [\mathbf{y}, \psi]$, which,  using the construction of products and equalizers, can be verified to be the class $[\mathbf{x}\mathbf{z}, \exists \mathbf{y} (\theta(\mathbf{x}, \mathbf{y}) \wedge \theta(\mathbf{z}, \mathbf{y}))]$. Then, as can be easily checked, the provability of the stated sequent is equivalent, by 1), to the fact that the diagonal morphism from $[\mathbf{x}, \phi]$ to this kernel pair is an isomorphism, which is in turn equivalent to the fact that $[\mathbf{x}\mathbf{y}, \theta]$ is a monomorphism.\\

Finally, suppose we have a monomorphism $[\mathbf{x}\mathbf{y}, \theta]: [\mathbf{x}, \psi] \rightarrowtail [\mathbf{y}, \phi]$. By 1), the morphism $[\mathbf{x}\mathbf{y}, \theta]: [\mathbf{x}, \psi] \rightarrow [\mathbf{y}, \exists \mathbf{x} \theta(\mathbf{x}, \mathbf{y})]$ is an isomorphism. Then, composing its inverse with the original monomorphism we have a subobject of the stated form, where $\psi(\mathbf{y})$ is the formula $\exists \mathbf{x} \theta(\mathbf{x}, \mathbf{y})$. Now, two subobjects $[\mathbf{y}, \psi], [\mathbf{y}, \eta]$ of $[\mathbf{y}, \phi]$ satisfy $[\mathbf{y}, \psi] \leq [\mathbf{y}, \eta]$ if and only if there exists a monomorphism $[\mathbf{y}, \psi] \rightarrowtail [\mathbf{y}, \eta]$, which by the previous argument must have the form $[\psi' \wedge (\mathbf{x}=\mathbf{y})]: [\mathbf{x}, \psi'] \rightarrowtail [\mathbf{y}, \eta]$ for some $\psi'$. But then, since $\psi$ and $\psi'$ must be  provable equivalent, this is a valid morphism if and only if the sequent $\psi(\mathbf{y}) \vdash_{\mathbf{y}} \eta(\mathbf{y})$ is provable in $\theory$. This completes the proof of 3).\end{proof}

To construct the desired model $M_{\theory}$ in the syntactic category of $\theory$, note that there is a natural $\Sigma$-structure assigning to the sort $A$ the formula $[x, \top]$ where $x$ is a variable of sort $A$, and to the relation symbols $R$ over variables $\mathbf{x}=x_1, ..., x_{\alpha}, ...$ of sorts $A_, ..., A_{\alpha}, ...$ respectively, the subobject $[\mathbf{x}, R(x_1, ..., x_{\alpha}, ...)] \rightarrowtail [\mathbf{x}, \top]$. We have now finally gotten to the important relationship between syntactic provability and semantic validity in $M_{\theory}$:

\begin{proposition}\label{thmp2} The sequent $\phi(\mathbf{x}) \vdash_{\mathbf{x}} \psi(\mathbf{x})$ is satisfied by the $\Sigma$-structure $M_{\theory}$ if and only if it is provable in $\theory$. Consequently, a formula $\eta(\mathbf{x})$ has full extension in $M_{\theory}$ if and only if it is provable in $\theory$.
\end{proposition}

\begin{proof} By definition, the stated sequent is satisfied by $M_{\theory}$ if and only if the corresponding subobjects in the interpretation satisfy $[\![\mathbf{x}, \phi]\!] \leq [\![\mathbf{x}, \psi]\!]$. By the construction of $M_{\theory}$, a straightforward induction on the complexity of $\phi$ proves that the interpretation $[\![\mathbf{x}, \phi]\!]$ is the subobject $[\mathbf{x}, \phi] \rightarrowtail [\mathbf{x}, \top]$. Therefore, the assertion $[\![\mathbf{x}, \phi]\!] \leq [\![\mathbf{x}, \psi]\!]$ is equivalent to the fact that the two subobjects $[\mathbf{x}, \phi], [\mathbf{x}, \psi]$ of $[\mathbf{x}, \top]$ satisfy $[\mathbf{x}, \phi] \leq [\mathbf{x}, \psi]$, which, by Lemma \ref{lemmap1} 3), is in turn equivalent to the fact that $\phi(\mathbf{x}) \vdash_{\mathbf{x}} \psi(\mathbf{x})$ is provable in $\theory$.\end{proof}

Proposition \ref{thmp2} says in a way that the model $M_{\theory}$ reflects all syntactical relations in the theory $\theory$; therefore, the analysis of categorical properties of $M_{\theory}$ will reveal facts about provability in $\theory$. 

We now have:

\begin{proposition}\label{catcomp}
 If $\theory$ is a $\kappa$-coherent (resp. $\kappa$-Heyting) theory, then $\mathcal{C}_{\theory}$ is a $\kappa$-coherent (resp. $\kappa$-Heyting) category.
\end{proposition}

\begin{proof} To prove $\mathcal{C}_{\theory}$ has $\kappa$-limits it suffices to prove it has $\kappa$-products and equalizers. As the product of $\gamma$-many objects $[\mathbf{x_i}, \phi_i]_{i<\gamma}$ (where the $\mathbf{x_i}$ are assumed to be disjoint) we can take the class $[\bigcup_{i<\gamma}\mathbf{x_i}, \bigwedge_{i<\gamma} \phi_i]$ together with the projections indicated below:

\begin{displaymath}
\xymatrix{
 [\mathbf{z}, \chi] \ar@{-->}[dd]_{[\mathbf{z} \cup (\bigcup_{i<\gamma}\mathbf{x_i}), \bigwedge_{j<\gamma} \theta_j]} \ar@{->}[ddrrrrrr]^{[\mathbf{z}\mathbf{x_j'}, \theta_j]} & & & & & &\\
& & & & & & \\
 [\bigcup_{i<\gamma}\mathbf{x_i}, \bigwedge_{i<\gamma} \phi_i] \ar@{->}[rrrrrr]_{[\bigcup_{i<\gamma}\mathbf{x_i}\mathbf{x_j'}, \bigwedge_{i<\gamma} \phi_i \wedge (\mathbf{x_j'}=\mathbf{x_j})]} & & & & & & [\mathbf{x_j'}, \phi_j]\\
}
\end{displaymath}
\\
Given morphisms $[\mathbf{z}\mathbf{x_j'}, \theta_j]$, the induced morphism into the product is given by the class $[\mathbf{z}\cup (\bigcup_{i<\gamma}\mathbf{x_i}), \bigwedge_{j<\gamma} \theta_j]$, since it can be easily verified that this is the only morphism that makes the diagram commute.

For the equalizer of a parallel pair of morphisms $[\mathbf{x}\mathbf{y}, \theta], [\mathbf{x}\mathbf{y}, \delta]$, we take:

\begin{displaymath}
\xymatrix{
[\mathbf{x'}, \exists \mathbf{y} (\theta(\mathbf{x'}, \mathbf{y}) \wedge \delta(\mathbf{x'}, \mathbf{y}))] \ar@{->}[rrrrr]^{[\mathbf{x'}\mathbf{x}, \exists \mathbf{y} (\theta \wedge \delta \wedge (\mathbf{x'}=\mathbf{x}))]} & & & & & [\mathbf{x}, \phi] \ar@/^{1pc}/[rr]^{[\mathbf{x}\mathbf{y}, \theta]} \ar@/_{1pc}/[rr]_{[\mathbf{x}\mathbf{y}, \delta]} & & [\mathbf{y}, \psi]\\
 & & & & & & & & \\
 & & & & & & & & \\
[\mathbf{z}, \chi] \ar@{->}[uuurrrrr]_{[\mathbf{z}\mathbf{x}, \eta]} \ar@{-->}[uuu]^{[\mathbf{z}\mathbf{x'}, \eta]} & & & & & & & & \\
}
\end{displaymath}
\\
and the universal property is satisfied with the indicated induced morphism. This proves that $\mathcal{C}_T$ has $\kappa$-limits. Note as well that there is an initial object given by $[\{\}, \bot]$, and a terminal object given by $[\{\}, \top]$.

To prove that the category has image factorizations, given a morphism $[\mathbf{x}\mathbf{y}, \theta]: [\mathbf{x}, \phi] \rightarrow [\mathbf{y}, \psi]$ we take its image as the subobject $[\mathbf{y}, \exists \mathbf{x} (\theta)] \rightarrowtail [\mathbf{y}, \psi]$. In particular, $[\mathbf{x}\mathbf{y}, \theta]$ is a cover if and only if the sequent $\psi(\mathbf{y}) \vdash_{\mathbf{y}} \exists \mathbf{x} \theta(\mathbf{x}, \mathbf{y})$ is provable in $\theory$. Then, from the construction of limits above, it can be verified straightforwardly using Frobenius axiom that covers are stable under pullbacks. To prove that the category has unions, take subobjects $[\mathbf{x}, \phi_i]_{i<\gamma}$ of $[\mathbf{x}, \phi]$ and define their union to be $[\mathbf{x}, \bigvee_{i<\gamma} \phi_i]$; the small distributivity law then ensures that unions are stable under pullback.

The validity of the transfinite transitivity property is proven using the transfinite transitivity rule, the construction of limits and the fact that $([\mathbf{x}_i\mathbf{y}, \theta_i]: [\mathbf{x}_i, \phi_i] \rightarrow [\mathbf{y}, \psi])_{i<\gamma}$ are jointly covering if and only if the sequent $\psi(\mathbf{y}) \vdash_{\mathbf{y}} \bigvee_{i<\gamma} \exists \mathbf{x}_i \theta_i(\mathbf{x}_i, \mathbf{y})$ is provable in $\theory$. It is necessary to compute the limit of a $\kappa$-chain $([\mathbf{x_{\alpha+1}}\mathbf{x_{\alpha}}, \theta_{\alpha}]: [\mathbf{x_{\alpha+1}}, \phi_{\alpha+1}] \to [\mathbf{x_{\alpha}}, \phi_{\alpha}])_{\alpha<\gamma}$. This can be computed using the construction of limits with products and equalizers; in this case, one can verify that the limit of such a chain reduces to compute the equalizer of the following diagram:

 \begin{displaymath}
\xymatrix{
[\mathbf{x}, \bigwedge_{\alpha<\gamma} \phi_{\alpha}] \ar@/^{1pc}/[rrrrrrrr]^{[\mathbf{x}\mathbf{x}', \bigwedge_{\alpha<\gamma} \phi_{\alpha} \wedge \bigwedge_{\alpha<\gamma} x_{\alpha}=x_{\alpha}']} \ar@/_{1pc}/[rrrrrrrr]_{[\mathbf{x}\mathbf{x}', \bigwedge_{\alpha<\gamma} \phi_{\alpha} \wedge \bigwedge_{\alpha<\gamma} \exists y_{\alpha+1}  \theta(y_{\alpha+1}, x_{\alpha}') \wedge y_{\alpha+1}=x_{\alpha+1}]} & & & & & & & & [\mathbf{x'}, \bigwedge_{\alpha<\gamma} \phi_{\alpha}]\\
}
\end{displaymath}

\noindent where $\mathbf{x}=x_0, ...x_{\alpha}, ...$. From this construction and the construction of equalizers in the syntactic category we can derive the sequent that expresses that the transfinite composition of jointly covering families are jointly covering, and verify that the sequent is provable within the theory making use of the transfinite transitivity rule.

Finally, if the theory is $\kappa$-Heyting, to construct universal quantification along a morphism $[\mathbf{x}\mathbf{y}, \theta]: [\mathbf{x}, \phi] \rightarrow [\mathbf{y}, \psi]$, take a subobject $[\mathbf{x}, \eta]$ of its domain, in the canonical form given in Lemma \ref{lemmap1} 3). Then define $\forall_{[\mathbf{x}\mathbf{y}, \theta]}([\mathbf{x}, \eta])$ to be the subobject $[\mathbf{y}, \psi \wedge \forall \mathbf{x} (\theta \to \eta)]$ of $[\mathbf{y}, \psi]$. It follows from Lemma \ref{lemmap1} 3) that this works.

This concludes the proof.\end{proof}

Due to the need of distinguishing several strengths of completeness, we introduce now the following terminology:

\begin{defs}
 A theory in a fragment of $\kappa$-first-order logic is said to be complete with respect to models of a certain type if whenever a sequent is valid in those models, it is provable from the axioms of the theory.
\end{defs}

With this definition, we now get:

\begin{thm}\label{cc}
 If $\kappa$ is any inaccessible cardinal, $\kappa$-coherent (resp. $\kappa$-Heyting) theories are semantically complete with respect to models in $\kappa$-coherent (resp. $\kappa$-Heyting) categories. 
\end{thm}

\subsection{Joyal's theorem and Morleyization}

The construction of the syntactic category is an aspect of the philosophy of theories as categories, which is supplemented by the concept of internal theory of a given category and the functorial semantics associated with it. For, say, a coherent category \synt{C}\ there is a canonical signature and coherent axioms associated to the category in such a way that coherent models of this theory correspond to coherent functors having the category as a domain. That is, functors which preserve the categorical properties are seen as models of the internal theory of the categories in the codomain categories. Moreover, model homomorphisms correspond in this view to natural transformations of functors. This allows us to think, for example, of the category $\cat{M}$ of set-valued coherent models of a theory as corresponding functors from the syntactic category of the theory to the category \Sets\ of sets. Consider now the further functor category $\Sets^{\cat{M}}$. To each coherent formula in context we can assign its extension in each of the models of $\cat{M}$, or equivalently, evaluate the models, seen as functors, on the corresponding object represented by the formula. This assignment is in fact functorial, and thus each coherent formula in context gives rise to a functor in $\Sets^{\cat{M}}$, which we call the evaluation functor at the corresponding formula. If we do this for every coherent formula in context, the assignment of evaluation functors at formulas is itself functorial, and gives rise to a functor $ev :\ \synt{C}{\theory}\to \Sets^{\cat{M}}$. 

In its original version, Joyal's theorem is a statement over ZFC which could be described as follows:

\begin{thm}\label{joyal}(Joyal)
Let \theory\ be a coherent theory and let $\cat{M}$ be the category of coherent models of $\theory$. Then the functor
\[ev :\ \synt{C}{\theory}\to \Sets^{\cat{M}}\]
is conservative and preserves any right adjoint to pullback functors that might exist in $\synt{C}{\theory}$.
\end{thm}

For the proof of Joyal's theorem we refer to \cite{mr}, Ch. 6, pp. 189, since we will later study and prove an infinitary version. The significance of the theorem resides in that it encapsulates three different completeness theorems. The conservativity of $ev$ is a categorical way of saying that models in $\cat{M}$ are semantically complete for coherent logic. In the particular case in which the logic is classical, this is precisely G\"odel's completeness theorem for first order logic. But even when we consider intuitionistic logic, the preservation of the right adjoint entails that $ev$ preserves the first-order structure of $\synt{C}{\theory}$, and through categorical semantics in the presheaf category $\Sets^{\cat{M}}$ we can see that the conservative embedding provides a universal Kripke model of the theory, resulting thus in Kripke completeness theorem for first-order intuitionistic logic. We shall go some steps further and consider variations that provide new completeness theorems in the infinitary case. More especifically, the work in Proposition \ref{scohcomp} and Theorem \ref{bt} will allow us to prove the following infinitary version of Joyal's theorem:

\begin{thm}
Let $\kappa$ be a weakly (resp. strongly) compact cardinal; let \theory\ be a $\kappa$-coherent theory of cardinality at most $\kappa$ (resp. of arbitrary cardinality) and let $\cat{M}$ be the category of $\kappa$-coherent models of $\theory$ of cardinality at most $\kappa$ (resp. of arbitrary cardinality). Then the functor
\[ev :\ \synt{C}{\theory}\to \Sets^{\cat{M}}\]
is conservative and preserves any right adjoint to pullback functors that might exist in $\synt{C}{\theory}$.
\end{thm}

We will see that the infinitary version of Joyal's theorem subsumes the completeness of $\kappa$-coherent logic with respect to \Sets-valued models (the conservativity of $ev$), that of $\kappa$-first-order classical logic (the particular case when $\synt{C}{\theory}$ is Boolean), and the completeness of $\kappa$-first order intuitionistic logic with respect to infinitary Kripke semantics (the universal Kripke model given by the embedding into the presheaf category).

We are also going to need one more fragment to work with:

\begin{defs}
 The $\kappa$-regular fragment is the fragment of $\kappa$-coherent logic that drops the disjunction $\vee$ from the language, and hence drops the rules involving it, but keeps the rule of dependent choice.
\end{defs}

The internal $\kappa$-coherent theory of a, say, $\kappa$-Heyting category can alternatively be described by a different axiomatization, which will be simpler for our purposes. Following \cite{johnstone}, where the process of rewriting of a classical first-order theory as an equivalent coherent theory is referred to as ``Morleyization'', we will also call ``Morleyizing'' a theory, in general, rewriting it into a theory in a less expressive fragment. From a categorical viewpoint (as opposed to the standard syntactic point of view), the syntactic category \synt{C}{\theory} of, for example, an intuitionistic $\kappa$-first-order theory \theory\, which is a $\kappa$-Heyting category, is also a $\kappa$-coherent (resp. $\kappa$-regular) category, and thus \synt{C}{\theory} has an internal $\kappa$-coherent theory (resp. internal $\kappa$-regular theory), which we refer to as ``the theory of $\kappa$-coherent (resp. $\kappa$-regular) models of \theory'', (its ``Morleyization'' $\theory^m$). It is not difficult to see (as it will become evident from the definition) that the theory and its Morleyization have equivalent syntactic categories:
\[\synt{C}{\theory}\simeq \synt{C}{\theory^m}\]
Although for classical $\kappa$-first-order theories, the $\kappa$-coherent Morleyization will have the same models in all Boolean $\kappa$-coherent categories, in general when Morleyizing a $\kappa$-first-order theory to a $\kappa$-coherent one (or a $\kappa$-coherent theory to a $\kappa$-regular one), this is not the case, but there still some gain in considering the category of models of the Morleyized theory, as our adaptation of Joyal's theorem will show. 

\begin{defs}
The theory of $\kappa$-regular ((i)--(iv) below) and the theory of $\kappa$-coherent ((i)--(v)) models $\theory^m$ of a $\kappa$-first-order (resp. $\kappa$-coherent) theory \theory\ over a signature $\Sigma$ is defined as follows: its signature $\Sigma^m$ extends $\Sigma$ by adding for each $\kappa$-first-order (resp. $\kappa$-coherent) formula $\phi$ over $\Sigma$ with free variables \alg{x} the relation symbol $P_{\phi}(\alg{x})$; then $\theory^m$ is the theory axiomatized by the following axioms:

\begin{enumerate}[(i)]

\item $P_{\phi}\dashv\vdash_{\alg{x}}\phi$ for every atomic formula $\phi$
\item $P_{\phi}\vdash_{\alg{x}}P_{\psi}$ for every sequent $\phi\vdash_{\alg{x}}{\psi}$ provable in \theory;
\item $P_{\bigwedge_{i<\gamma}\phi_i}\dashv\vdash_{\alg{x}}\bigwedge_{i<\gamma}P_{\phi_i}$;
\item $P_{\exists{\mathbf{y}}\phi}\dashv\vdash_{\alg{x}}\exists{\mathbf{y}}P_{\phi}$;
\item $P_{\bigvee_{i<\gamma}\phi_i}\dashv\vdash_{\alg{x}}\bigvee_{i<\gamma}P_{\phi_i}$.

\end{enumerate} 
\end{defs}

The theory of $\kappa$-coherent models of a positive $\kappa$-coherent\footnote{The fragment which results after discarding $\bot$.} theory is defined similarly; alternatively (since we are only discarding $\bot$) we could also treat $\bot$ as a propositional variable and add the axioms 
\[\bot\vdash_{\alg{x}}\phi\]
for all formulas $[\mathbf{x}, \phi]$ in context. 

\begin{defs}
 We will say that a positive $\kappa$-coherent model of a $\kappa$-coherent theory is \emph{possibly exploding}, and make the convention that such a model is \emph{exploding} if it assigns $\bot$ the value true.
\end{defs}

Note that since $P_{\phi}\vdash_{\alg{x}}P_{\psi}$ in $\theory^m$ if and only if $\phi\vdash_{\alg{x}}{\psi}$ in \theory, if positive $\kappa$-coherent theories are complete for $\Sets$-valued models, then $\kappa$-coherent theories will be complete for modified (i.e., possibly exploding) $\Sets$-valued models. Incidentally, any model of $\theory^m$ that assigns $\bot$ the value true must be inhabited, since $\bot\vdash \fins{x} (x=x)\in \theory^m$.

\subsection{Beth and Kripke models}

The following is a direct generalization to the infinitary case of the Beth models (see \cite{beth}) used in intuitionistic logic, except that we allow the nodes of the underlying tree to be possibly exploding (i.e., to force $\bot$):

\begin{defs}\label{bethmodelt}
 A Beth model for pure $\kappa$-first-order logic over $\Sigma$ is a quadruple $\mathcal{B}=(K, \leq, D, \Vdash)$, where $(K, \leq)$ is a tree of height $\kappa$ and levels of size less than $\kappa$, and with a set $B$ of branches (i.e., maximal chains in the partial order) each of size $\kappa$; $D$ is a set-valued functor on $K$, and the forcing relation $\Vdash$ is a binary relation between elements of $K$ and sentences of the language with constants from $\bigcup_{k \in K}D(k)$, defined recursively for formulas $\phi$ as follows. There is an interpretation of function and relation symbols in each $D(k)$; if $R_k \subseteq D(k)^{\lambda}$ is the interpretation in $D(k)$ of the $\lambda$-ary relation symbol $R$ in the language, we have $k \leq l \implies R_k(D_{kl}(\mathbf{c})) \subseteq R_l(\mathbf{c})$ for $\mathbf{c} \subseteq D_k$, and: 
 
 \begin{enumerate}
  \item $k \Vdash R(\mathbf{s}(\mathbf{d})) \iff \exists \alpha<\kappa \forall b \in B_k \exists l \in b, level(l)=level(k)+\alpha$ $(R_l(\mathbf{s}(D_{kl}(\mathbf{d}))))$
  \item $k \Vdash \bigwedge_{i<\gamma}\phi_i(\mathbf{d}) \iff k \Vdash \phi_i(\mathbf{d}) \text{ for every } i<\gamma$
  \item $k \Vdash \bigvee_{i<\gamma}\phi_i(\mathbf{d}) \iff \exists \alpha<\kappa \forall b \in B_k \exists l \in b, level(l)=level(k)+\alpha \qquad (l \Vdash \phi_i(D_{kl}(\mathbf{d})) \text{ for some } i<\gamma)$
  \item $k \Vdash \phi(\mathbf{d}) \to \psi(\mathbf{d'}) \iff \forall k' \geq k (k' \Vdash \phi(D_{kk'}(\mathbf{d})) \implies k' \Vdash \psi(D_{kk'}(\mathbf{d'}))$
  \item $k \Vdash \exists \mathbf{x} \phi(\mathbf{x}, \mathbf{d}) \iff \exists \alpha<\kappa \forall b \in B_k \exists l \in b, level(l)=level(k)+\alpha \qquad \exists \mathbf{e} \subseteq D(l) (l \Vdash \phi(\mathbf{e}, D_{kl}(\mathbf{d}))$
  \item $k \Vdash \forall \mathbf{x} \phi(\mathbf{x}, \mathbf{d}) \iff \forall k' \geq k \forall \mathbf{e} \subseteq D_{k'} (k' \Vdash \phi(\mathbf{e}, D_{kk'}(\mathbf{d})))$
 \end{enumerate}

A Beth model for a theory \theory\ is a Beth model for $\kappa$-first-order logic forcing all the axioms of the theory and not forcing $\bot$. 
\end{defs}

We have now:

\begin{proposition}  
$\kappa$-first-order logic is sound for Beth models.
\end{proposition}

\begin{proof}
The key part of the proof is to note that the following property holds: for any $\kappa$-first-order formula $\phi(\mathbf{x})$ and any node $k$ in the Beth model, we have $k \Vdash \phi(\mathbf{c}) \iff \exists \alpha<\kappa \forall b \in B_k \exists l \in b, level(l)=level(k)+\alpha \qquad (l \Vdash \phi(D_{kl}(\mathbf{c})))$. This in turn can be easily proved by induction on the complexity of $\phi$ making use of the regularity of $\kappa$. Using now this property, it is easy to check the validity of all axioms and rules of $\kappa$-first-order logic; the regularity of $\kappa$ is used in the soundness of the transfinite transitivity rule.
\end{proof}

The restriction on the height and the size of the levels of the Beth model was motivated by proving soundness for full $\kappa$-first-order logic. This restriction can be relaxed if we do not intend to do so. We make thus the following:

\begin{defs} 
Given a regular cardinal $\delta<\kappa$ and a set \theory\ consisting of logical and non-logical axioms in a $\kappa$-first-order language, a partial Beth model of height $\delta$ for \theory\ is defined like a Beth model for $\theory$, except that the branches have size $\delta$, the ordinal $\alpha$ in the clauses for atomic formulas, disjunction and existential quantification satisfies $\alpha<\delta$, and the model only forces axioms in $\theory$.
\end{defs}

A Kripke model is a special kind of Beth model none of whose nodes forces $\bot$ and where the forcing relation for atomic formulas, disjunction and existential quantification satisfies the stronger condition $level(l)=level(k)$:

\begin{defs}
 A Kripke model for pure first-order logic over $\Sigma$ is a quadruple $\mathcal{K}=(K, \leq, D, \Vdash)$, where $(K, \leq)$ is a tree, $D$ is a set-valued functor on $K$ and the forcing relation $\Vdash$ is a binary relation between elements of $K$ and sentences of the language with constants from $\bigcup_{k \in K}D(k)$, satisfying $k \nVdash \bot$ and defined recursively for formulas $\phi$ as follows. There is an interpretation of function and relation symbols in each $D(k)$; if $R_k \subseteq D(k)^{\lambda}$ is the interpretation in $D(k)$ of the $\lambda$-ary relation symbol $R$ in the language, we have $k \leq l \implies R_k(D_{kl}(\mathbf{c})) \subseteq R_l(\mathbf{c})$ for $\mathbf{c} \subseteq D_k$, and: 
 
 \begin{enumerate}
  \item $k \Vdash R(\mathbf{s}(\mathbf{d})) \iff R_k(\mathbf{s}(\mathbf{d}))$
  \item $k \Vdash \bigwedge_{i<\gamma}\phi_i(\mathbf{d}) \iff k \Vdash \phi_i(\mathbf{d}) \text{ for every } i<\gamma$
  \item $k \Vdash \bigvee_{i<\gamma}\phi_i(\mathbf{d}) \iff k \Vdash \phi_i(\mathbf{d}) \text{ for some } i<\gamma$
  \item $k \Vdash \phi(\mathbf{d}) \to \psi(\mathbf{d'}) \iff \forall k' \geq k (k' \Vdash \phi(D_{kk'}(\mathbf{d})) \implies k' \Vdash \psi(D_{kk'}(\mathbf{d'})))$
  \item $k \Vdash \exists \mathbf{x} \phi(\mathbf{x}, \mathbf{d}) \iff \exists \mathbf{e} \subseteq D(k) (k \Vdash \phi(\mathbf{e}, \mathbf{d}))$
  \item $k \Vdash \forall \mathbf{x} \phi(\mathbf{x}, \mathbf{d}) \iff \forall k' \geq k \forall \mathbf{e} \subseteq D_{k'} (k' \Vdash \phi(\mathbf{e}, D_{kk'}(\mathbf{d})))$
 \end{enumerate}

 A Kripke model for a theory \theory\ is a Kripke model forcing all the axioms of the theory. 
\end{defs}

A Kripke model can also be seen categorically as a model on a presheaf category. That is, if \synt{C}{\theory}\ is the syntactic category of the theory, a Kripke model on $(K, \leq)$ is nothing but a $\kappa$-Heyting functor $F: \synt{C}{\theory}\to \Sets^{K}$, since such a functor determines the set-valued functor $D=F([x, \top]): K \to \Sets$ which specifies the underlying domains of the nodes. In this case the forcing relation is given by $k \Vdash \phi(\mathbf{d})$ for $\mathbf{d} \in D(k)^n$ if and only if $\mathbf{d}: [-, k] \to D^n=F([\mathbf{x}, \top])$ factors through the subobject $F([\mathbf{x}, \phi]) \rightarrowtail F([\mathbf{x}, \top])$, where we use Yoneda lemma to identify elements of $D(k)$ with natural transformations $[-, k] \to D$. This definition is precisely the forcing relation for the Kripke-Joyal semantics in the topos $\Sets^{K}$, whence the name Kripke associated to it. We therefore have:

\begin{proposition}  
$\kappa$-first-order logic is sound for Kripke models.
\end{proposition}

\begin{proof}
It is enough to note that presheaf categories are $\kappa$-Heyting and apply Theorem \ref{cc}.
\end{proof}

More generally, one can consider Kripke models on arbitrary categories $\mathcal{M}$ instead of the tree $K$, and it turns out that the semantics of the Kripke model over $\mathcal{M}$ can be recovered in terms of Kripke semantics over a certain collection of trees. To do that, consider first the poset $P$ which consists of finite composable sequences of morphisms of $\mathcal{M}$, i.e., chains $A_0 \to ... \to A_n$ in $\mathcal{M}$. One such sequence is below another in $P$ if the former is an initial segment of the latter. There is a functor $E:P \to \mathcal{M}$ sending each chain to the last object in it and sending any morphism $f$ of $P$ to the composite of the morphisms of $\mathcal{M}$ that are in the codomain minus the domain of $f$. Now, given a Kripke model $F: \synt{C}{\theory}\to \Sets^{\mathcal{M}}$, we can compose $F$ with the transpose $E^*: \Sets^{\mathcal{M}} \to \Sets^{P}$, and if this latter is a conservative $\kappa$-Heyting functor, this will provide a Kripke model on $P$ forcing precisely the same formulas as the original model. Finally, the Kripke model on $P$ can be regarded as a collection of Kripke models on trees, where the roots of the trees are given by one-element chains. This construction amounts to build the Diaconescu cover of the topos $\Sets^{\mathcal{M}}$ (see e.g. \cite{maclane-moerdijk}). In our case the discussion above shows that for our purposes it is enough to prove the following, which is the infinitary counterpart of section 1.744 of \cite{fs}:

\begin{lemma}
 The functor $E^*: \Sets^{\mathcal{M}} \to \Sets^{P}$ is conservative and $\kappa$-Heyting.
\end{lemma}

\begin{proof}
 The conservativity of $E^*$ follows from the fact that $E$ is surjective on objects and arrows. To prove that it is $\kappa$-Heyting, the non-trivial part is proving that it preserves $\forall$. For a natural transformation $f: F \to G$ in $\mathcal{S}et^{\mathcal{M}}$ and a subfunctor $A$ of $F$, we need to show that $E^*(\forall_fA)$ is the same subfunctor of $E^*(G)$ as $\forall_{E^*(f)}E^*(A)$. By definition, for any object $p$ in $P$ and $y \in E^*(G)(p)=G(E(p))$, we have $y \in \forall_{E^*(f)}E^*(A)(p)$ if and only if for all arrows $l: p \to q$ in $P$ one has:

$$E^*(f)_q^{-1}(G(E(l))(y)) \subseteq E^*(A)(q)$$

$$\iff \forall x (E^*(f)_q(x)=G(E(l))(y) \implies x \in E^*(A)(q))$$

$$\iff \forall x (f_{E(q)}(x)=G(E(l))(y) \implies x \in A(E(q))) \qquad (1)$$

On the other hand, also by definition, for $y \in G(E(p))$ one has $y \in \forall_fA(E(p))$ if and only if for all arrows $t: E(p) \to r$ in $\mathcal{M}$ one has:

$$f_r^{-1}(G(t)(y)) \subseteq A(r)$$

$$\iff \forall x (f_r(x)=G(t)(y) \implies x \in A(r)) \qquad (2)$$

But because the functor $E$ is surjective (both on objects and arrows), we can find $q, l \in P$ such that $r=E(q)$ and $t=E(l)$, from which we deduce that $(1)$ and $(2)$ above are equivalent. Hence, $E^* \forall=\forall E^*$, as we wanted.
 
\end{proof}

\subsection{Syntactic sites}

The syntactic categories for fragments of $\kappa$-first-order logic can be equipped with appropriate Grothendieck topologies in such a way that the corresponding sheaf toposes are conservative models of the corresponding theories. Given a $\kappa$-regular category, we can define the $\kappa$-regular coverage, where the covering families are all singletons $f$ where $f$ is a cover. Similarly, for a $\kappa$-coherent category we can define the $\kappa$-coherent coverage, where the covering families are given by families of arrows $f_i: A_i \to A$ of cardinality less than $\kappa$ such that the union of their images is the whole of $A$ (in particular, the initial object $0$ is covered by the empty family). We can also find (see \cite{bj}) a conservative sheaf model given by Yoneda embedding into the sheaf topos obtained with the $\kappa$-coherent coverage. As proven in \cite{bj}, the embedding preserves $\kappa$-unions and $\kappa$-intersections, as well as any Heyting structure that might exist in $\mathcal{C}$. To highlight the fact that images and unions are stable under pullback is crucial, we prove the following lemma, which can be regarded as a generalization of the result corresponding to the finitary case:

\begin{lemma}\label{shemb}
 Given a $\kappa$-coherent (resp. $\kappa$-Heyting) category $\mathcal{C}$ with the $\kappa$-coherent coverage $\tau$, Yoneda embedding $y: \mathcal{C} \to \mathcal{S}h(\mathcal{C}, \tau)$ is a conservative $\kappa$-coherent (resp. $\kappa$-Heyting) functor and $\mathcal{S}h(\mathcal{C}, \tau)$ is a $\kappa$-Heyting category.
\end{lemma}

\begin{proof}
 By \cite{mr}, Proposition 3.3.3, we know that all representable functors are sheaves for the $\kappa$-coherent coverage, since the fact that the union of the images of the arrows in a covering family over $A$ is the whole of $A$ is equivalent to the fact that the family is effective epimorphic, and this is precisely the sheaf condition on representable functors. 
 
 In case $\mathcal{C}$ is $\kappa$-Heyting, the embedding preserves universal quantification as shown in \cite{bj}, Lemma 3.1. For the sake of completeness we reproduce here the proof. Let $f: A \to B$ be a morphism of $\mathcal{C}$, $A' \rightarrowtail A$ a subobject of A, and write $B' \rightarrowtail B$ for $\forall_f(A' \rightarrowtail A)$. Let $R' \rightarrowtail y(B)$ be any subobject of $y(B)$ in $\mathcal{S}h(\mathcal{C}, \tau)$ such that $(y(f))^*(R) \leq y(A')$ in $\mathcal{S}ub(y(A))$; we must show that $R \leq y(B’)$, i.e. that every morphism in $R$ factors through $B' \rightarrowtail B$. Let $g: C \to B$ be such a morphism, and let $h: D \to A$ be its pullback along $f$; then $h \in f^*(R)$, and so $h$ factors through $A' \rightarrowtail A$, that is, the image $A'' \rightarrowtail A$ of $h$ satisfies $A''<A'$. Since image factorizations in $\mathcal{C}$ are stable under pullback, it follows that the image $B'' \rightarrowtail B$ of $B$ satisfies $B'' \leq B'$; so we have our required factorization.
 
 Yoneda embedding preserves limits, and limits of sheaves are computed as in presheaves, so it remains to prove that it preserves images and $\kappa$-unions. Given a cover $f: A \twoheadrightarrow B$, we need to prove that $[-, A] \to [-, B]$ is a sheaf epimorphism, i.e., that it is locally surjective. For this it is enough to find a covering family over each object $C$ that witnesses the local surjectivity. Given an element $g$ in $[C, B]$, we can simply form the pullback of $f$ along $g$, obtaining thus a covering family over consisting on the single arrow $g^*(f)$ which will clearly witness the local surjectivity.
 
 The argument for the preservation of unions is similar: given the union $\bigvee_{i<\gamma}A_i$ of subobjects $f_i: A_i \to B$ we need to show that $[-, \bigvee_{i<\gamma}A_i]$ is the union of the sheaves $[-, A_i]$ . Given an object $C$ and an element $g$ in $[C, \bigvee_{i<\gamma}A_i]$, the pullbacks along $g$ of $f'_i: A_i \to \bigvee_{i<\gamma}A_i$ give a covering family $\{g^*(f'_i): P_i \to C\}_{i<\gamma}$ with the property that $g.g^*(f'_i) \in [P_i, \bigvee_{i<\gamma}A_i]$ belongs to $[P_i, A_i]$, which is enough to guarantee that $[-, \bigvee_{i<\gamma}A_i]$ is indeed the union of the $[-, A_i]$.  
 
 Finally, we show that the sheaf topos is a $\kappa$-coherent category by proving that the transfinite transitivity property holds in $\mathcal{S}h(\mathcal{C}, \tau)$. To this end, suppose we have a family of sheaves $\{S_{f}: \beta<\gamma, f \in \gamma^{\beta}\}$ satisfying the premises of the transfinite transitivity property, that is, that $\{S_{g} \to S_f: g \in \gamma^{\beta+1}, g|_{\beta}=f\}$ form a jointly covering family and that $S_{f|_{\beta}}=\lim_{\alpha<\beta}S_{f|_{\alpha}}$ for limit $\beta$. Then given $c \in S_{\emptyset}(C)$ we define by transfinite recursion a covering family $\{l_f: C_{f} \to C: \beta<\gamma, f \in \gamma^{\beta}\}$ such that, given $f \in \gamma^{\gamma}$, $c.l_f \in \bigwedge_{\alpha<\gamma}S_{f_i|_{\alpha}}(C)$ for some $f_i \in \gamma^{\gamma}$, witnessing that $\{\bigwedge_{\alpha<\gamma}S_{f|_{\alpha}} \to S_{\emptyset}: f \in \gamma^{\gamma}\}$ is a jointly covering family. In fact, the covering family over $C$ will be such that for any fixed $\beta<\gamma$ we will have that $\{l_f: C_{f} \to C: f \in \gamma^{\beta}\}$ is a witness of the joint covering of the sheaves $\{S_{f}: f \in \gamma^{\beta}\}$, that is, given $f \in \gamma^{\beta}$ we will have $c.l_f \in S_{f_i}(C)$ for some $f_i \in \gamma^{\beta}$. 
 
 Supposing that $\{l_f: C_{f} \to C: \beta<\mu, f \in \gamma^{\beta}\}$ has been defined, we show how to define the family at level $\mu$. If $\mu$ is a successor ordinal $\mu=\alpha+1$, we have by inductive hypothesis a covering $\{l_f: C_{f} \to C: f \in \gamma^{\alpha}\}$ such that, given $f \in \gamma^{\alpha}$, $c.l_f \in S_{f_i}(C)$ for some $f_i \in \gamma^{\alpha}$. Then, because $\{S_{g} \to S_{f_i}: g \in \gamma^{\mu}, g|_{\alpha}=f_i\}$ is jointly covering, we can find a covering $\{h_{gf_i}: C_g \to C_{f_i} : g \in \gamma^{\mu}, g|_{\alpha}=f_i\}$ such that, given $g \in \gamma^{\mu}, g|_{\alpha}=f_i$, $c.l_g=(c.l_{f_i}).h_{gf_i} \in S_{g_j}(C)$ for some $g_j \in \gamma^{\mu}$. This extends, by transitivity, the definition of the covering family to level $\mu$. If $\mu$ is a limit ordinal and $f \in \gamma^\mu$, we simply take $C_f$ to be the limit of the diagram formed by $C_{f|_{\alpha}}: \alpha<\mu$. Then clearly, given $f \in \gamma^\mu$, $c.l_f \in \bigwedge_{\alpha<\mu}S_{f_k|_{\alpha}}$ for some $f_k \in \gamma^\mu$. This finishes the recursive construction of the family over $C$ and proves the transfinite transitivity property for the sheaves.
 
\end{proof}

We get immediately:

\begin{thm}
 If $\kappa$ is any inaccessible cardinal, $\kappa$-first-order theories are complete with respect to models in $\kappa$-Heyting Grothendieck toposes. 
\end{thm}

\section{Completeness}

\subsection{Completeness of infinitary coherent logic}

We will need the following technical lemma, which corresponds to the canonical well-ordering of $\kappa \times \kappa$ from \cite{jechst}:

\begin{lemma}\label{dwo}
 For every cardinal $\kappa$ there is a well-ordering $f: \kappa \times \kappa \to \kappa$ with the property that $f(\beta, \gamma) \geq \gamma$.
\end{lemma}

\begin{proof}
 We define $f$ by induction on $\max (\beta, \gamma)$ as follows:
 
 $$f(\beta, \gamma)=\begin{cases} \sup\{f(\beta', \gamma')+1: \beta', \gamma'<\gamma\}+\beta  & \mbox{if }  \beta<\gamma  \\ \sup\{f(\beta', \gamma')+1: \beta', \gamma'<\beta\}+\beta+\gamma & \mbox{if } \gamma \leq \beta \end{cases}$$
 
\noindent which satisfies the required property (see \cite{jechst}, Theorem 3.5).
\end{proof}

We have now:

\begin{thm}\label{shcomp}
 Let $\kappa$ be an inaccessible cardinal. Then any $\kappa$-coherent theory of cardinality less than $\kappa$ has a partial Beth model of height less than $\kappa$.
\end{thm}

\begin{proof}
 Consider the syntactic category $\mathcal{C}_{\theory}$ of the theory and its conservative embedding in the topos of sheaves with the $\kappa$-coherent coverage, $\mathcal{C}_{\theory} \to \mathcal{S}h(\mathcal{C}_{\theory}, \tau)$. By assumption, the cardinality of the set $S$ of all subformulas of axioms of the theory is $\delta<\kappa$. Consider for any object $A$ the set of basic covering families over $A$ (which are given by jointly cover sets of arrows of cardinality less than $\kappa$) that witness than some subformula in $S$ is forced by $A$. That is, if $S$ contains a subformula $\phi$ which is a (nonempty) disjunction $\bigvee_{i<\gamma}\phi_i(\boldsymbol{\beta})$ (resp. an existential formula $\exists_{\alpha<\gamma}\mathbf{x}_{\alpha}\psi(\mathbf{x}_0, ...,\mathbf{x}_{\alpha}, ..., \boldsymbol{\beta})$), and $A \Vdash \phi(\boldsymbol{\beta})$, we include in the set of coverings one of the form $l_j: C_j \to A$, where for each $j$ we have $C_j \Vdash \phi_{i_j}(\boldsymbol{\beta} l_j)$ for some $i_j<\gamma$ (resp. $C_j \Vdash \psi(\boldsymbol{\beta_0^j}, ..., \boldsymbol{\beta_{\alpha}^j}, ... \boldsymbol{\beta} l_j)$ for some $\boldsymbol{\beta_0^j}, ..., \boldsymbol{\beta_{\alpha}^j}, ...$). In case $\phi$ is $\bot$, or $\phi$ is a conjunctive subformula, or $A \nVdash \phi(\boldsymbol{\beta})$ we just consider the identity arrow as a cover.
 
 The set of covering families over a given object $A$ just specified has thus cardinality $\delta$. By adding identity covers to each set we can assume without loss of generality that $\delta$ is regular and bigger than the maximum arity of function and relation symbols in $S$. Construct a functor from a tree of height $\delta$ to the syntactic category, defined recursively on the levels of the tree. Start with a well-ordering $f: \delta \times \delta \to \delta$ as in Lemma \ref{dwo}, i.e.,  with the property that $f(\beta, \gamma) \geq \gamma$. We describe by an inductive definition how the tree obtained as the image of the functor is constructed.
 
 The root of that tree is the terminal object. Suppose now that the tree is defined for all levels $\lambda<\mu$; we show how to define the nodes of level $\mu$. Suppose first that $\mu$ is a successor ordinal $\mu=\alpha+1$, and let $\alpha=f(\beta, \gamma)$. Since by hypothesis $f(\beta, \gamma) \geq \gamma$, the nodes $\{p_i\}_{i<m_{\gamma}}$ at level $\gamma$ are defined. Consider the morphisms $g_{ij}^{\alpha}$ over $p_i$ assigned to the paths from each of the nodes $p_i$ to the nodes of level $\alpha$. To define the nodes at level $\alpha+1$, take then the $\beta-th$ covering family over each $p_i$ and pull it back along the morphisms $g_{ij}^{\alpha}$. This produces covering families over each node at level $\alpha$, whose domains are then the nodes of level $\alpha+1$. Suppose now that $\mu$ is a limit ordinal. Then each branch of the tree of height $\mu$ already defined determines a diagram, whose limit is defined to be the node at level $\mu$ corresponding to that branch. 
 
 The tree has height $\delta$, and clearly, the morphisms assigned to the paths from any node $p$ till the nodes of level $\alpha$ in the subtree over $p$ form a basic covering family of $p$ because of the transfinite transitivity property. Define now a partial Beth model $B$ over this tree by defining as the underlying set of a node $q$ the set of arrows from $q$ to the object $[x, \top]$ in the syntactic category, and where the function between the underlying set of a node and its successor is given by composition with the corresponding arrow. There is an interpretation of the function symbols in the subset underlying each node which corresponds to composition with the interpretation in the category of the corresponding function symbol. For relations $R$ (including equality), we set by definition $R_q(\mathbf{s}(\boldsymbol{\alpha}))$ if and only if $q$ forces $R(\mathbf{s}(\boldsymbol{\alpha}))$ in the sheaf semantics of the topos, that is, if $q \Vdash R(\mathbf{s}(\boldsymbol{\alpha}))$ (we identify the category with its image through Yoneda embedding). We shall now prove the following:\\
 
$Claim:$ For every node $p$, every tuple $\boldsymbol{\alpha}$ and every formula $\phi \in S$, $p \Vdash \phi(\boldsymbol{\alpha})$ if and only if $p \Vdash_B \phi(\boldsymbol{\alpha})$.\\

The proof goes by induction on $\phi$.
 
 \begin{enumerate}
  \item If $\phi$ is atomic, the result is immediate by definition of the underlying structures on each node.
  
  \item If $\phi=\bigwedge_{i<\gamma}\psi_i$, the result follows easily from the inductive hypothesis, since we have $p \Vdash \bigwedge_{i<\gamma}\psi_i(\boldsymbol{\alpha})$ if and only if $p \Vdash \psi_i(\boldsymbol{\alpha})$ for each $i<\gamma$, if and only if $p \Vdash_B \psi_i(\boldsymbol{\alpha})$ for each $i<\gamma$, if and only if $p \Vdash_B \bigwedge_{i<\gamma}\psi_i(\boldsymbol{\alpha})$.
  
  \item Suppose $\phi= \bigvee_{i<\gamma} \psi_i$. If $p \Vdash \bigvee_{i<\gamma} \psi_i$, then there is a basic covering family $\{f_i: A_i \to p\}_{i<\lambda}$ that appears at some point in the well-ordering, such that for each $i<\lambda$, $A_i \Vdash \psi_{k_i}(\boldsymbol{\alpha} f_i)$ for some $k_i<\gamma$. Now this covering family is pulled back along all paths $g_j$ of a subtree to create the nodes of a certain level of the subtree over $p$. Hence, every node $m_j$ in such a level satisfies $m_j \Vdash \psi_{k_j}(\boldsymbol{\alpha} fg'_j)$ for some $k_j<\gamma$. By inductive hypothesis, $m_j \Vdash_B \psi_{k_j}(\boldsymbol{\alpha} fg'_j)$, and hence we have $p \Vdash_B \bigvee_{i<\gamma} \psi_i$.
  
  Conversely, if $p \Vdash_B \bigvee_{i<\gamma} \psi_i$, there is a level in the subtree over $p$ such that for every node $m_j$ there one has $m_j \Vdash_B \psi_{k_j}(\boldsymbol{\alpha} f_j)$ for some $k_j<\gamma$, so by inductive hypothesis $m_j \Vdash \psi_{k_j}(\boldsymbol{\alpha} f_j)$. Since $\{f_k: m_k \to p\}$ is, by construction, a basic covering family, we must have $p \Vdash \bigvee_{i<\gamma} \psi_i$.
   
  \item Suppose $\phi=\exists \mathbf{x} \psi(\mathbf{x}, \boldsymbol{\alpha})$. If $p \Vdash \exists \mathbf{x} \psi(\mathbf{x}, \boldsymbol{\alpha})$, then there is a basic covering family $\{f_i: A_i \to p\}_{i<\lambda}$ that appears at some point in the well-ordering, such that for each $i$ one has $A_i \Vdash \psi(\boldsymbol{\beta_i}, \boldsymbol{\alpha} f_i)$ for some $\boldsymbol{\beta_i}: A_i \to [\mathbf{x}, \top]$. This basic cover is hence pulled back along all paths $g_j$ of a subtree to create the nodes of a certain level of the subtree over $p$. The nodes $m_{ij}$ in this level will have the property that $m_{ij} \Vdash \psi(\boldsymbol{\beta_i}g'_j, \boldsymbol{\alpha} f_ig'_j)$, and hence, by inductive hypothesis, that $m_{ij} \Vdash_B \psi(\boldsymbol{\beta_i}g'_j, \boldsymbol{\alpha} f_ig'_j)$. By definition, we get thus $p \Vdash_B \exists \mathbf{x} \psi(\mathbf{x}, \boldsymbol{\alpha})$.
  
  Conversely, suppose that $p \Vdash_B \exists \mathbf{x} \psi(\mathbf{x}, \boldsymbol{\alpha})$. Then there is a level in the subtree over $p$ such that for every node $m_k$ there one has $m_k \Vdash_B \psi(\boldsymbol{\beta_k}, \boldsymbol{\alpha} f_k)$ for some $\boldsymbol{\beta_k}: m_k \to [\mathbf{x}, \top]$, and hence, by inductive hypothesis, such that $m_k \Vdash \psi(\boldsymbol{\beta_k}, \boldsymbol{\alpha} f_k)$. Since the arrows $f_k: m_k \to p$ form a basic cover of $p$, we must have $p \Vdash \exists \mathbf{x} \psi(\mathbf{x}, \boldsymbol{\alpha})$.
  
  \end{enumerate}
 
\end{proof}

\begin{proposition}\label{cohcomp}
If $\kappa$ is an inaccessible cardinal, $\kappa$-coherent theories of cardinality less than $\kappa$ are complete with respect to \Sets-valued models.
\end{proposition}

\begin{proof}
 It is enough to prove that every object in the sheaf model forcing the antecedent $\phi(\boldsymbol{\alpha})$ of a valid sequent $\phi \vdash_{\mathbf{x}} \psi$ also forces the consequent $\psi(\boldsymbol{\alpha})$ for every tuple $\boldsymbol{\alpha}$ in the domain. Construct a Beth model over a tree as above but taking as the root of the tree a given object forcing $\phi(\boldsymbol{\alpha})$ and including in the set of formulas $S$ also the subformulas of $\phi$ and $\psi$. For each branch $\mathbf{b}$ of the tree, consider the directed colimit $\mathbf{D_b}$ of all the underlying structures in the nodes of the branch, with the corresponding functions between them. Such a directed colimit is a structure under the definitions:
 
 \begin{enumerate}
 \item for each function symbol $f$, we define $f(\overline{x_0}, ..., \overline{x_{\lambda}}, ...)=\overline{f(x_0, ..., x_{\lambda}, ...)}$ for some representatives $x_i$ of $\overline{x_i}$; in particular, constants $\mathbf{c}$ are interpreted as $\overline{\mathbf{c}}=\overline{c_0}, ..., \overline{c_{\lambda}}, ...$;
 \item for each relation symbol $R$ we define $R(\overline{x_0}, ..., \overline{x_{\lambda}}, ...) \iff R(x_0, ..., x_{\lambda}, ...)$ for some representatives $x_i$ of $\overline{x_i}$.
 \end{enumerate} 
 
  It is easy to check, using the regularity of $\delta$, that the structure is well defined and that the choice of representatives is irrelevant. We will show that such a structure is a (possible exploding) positive $\kappa$-coherent model of the theory satisfying $\phi(\overline{\boldsymbol{\alpha}})$. Indeed, we have the following:
 
 $Claim:$ Given any $\kappa$-coherent formula $\phi(x_0, ..., x_{\lambda}, ...) \in S$, we have $\mathbf{D_b} \vDash \phi(\overline{\alpha_0}, ..., \overline{\alpha_{\lambda}}, ...)$ if and only if for some node $n$ in the path $\mathbf{b}$, the underlying structure $C_n$ satisfies $C_n \Vdash \phi(\alpha_0, ..., \alpha_{\lambda}, ...)$ for some representatives $\alpha_i$ of $\overline{\alpha_i}$.
 
 The proof of the claim is by induction on the complexity of $\phi$. 
 
 \begin{enumerate}
  \item If $\phi$ is $R(t_0, ..., t_{\lambda}, ...)$ or $s=t$ for given terms $t_i, s, t$, the result follows by definition of the structure. 
  
  \item If $\phi$ is of the form $\bigwedge_{i<\gamma} \theta_i$ the result follows from the inductive hypothesis: $\theta_i$ is forced at some node $n_i$ in the path $\mathbf{b}$, and therefore $\bigwedge_{i<\gamma} \theta_i$ will be forced in any upper bound of $\{n_i: i<\gamma\}$ (here we use the regularity of $\delta$).
  
  \item If $\phi$ is of the form $\bigvee_{i<\gamma} \theta_i$ and $\mathbf{D_b} \vDash \phi(\overline{\alpha_0}, ..., \overline{\alpha_s}, ...)$, then we can assume that $\mathbf{D_b} \vDash \theta_i(\overline{\alpha_0}, ..., \overline{\alpha_s}, ...)$ for some $i<\gamma$, so that by inductive hypothesis we get $C_n \Vdash \phi(\alpha_1, ..., \alpha_s, ...)$ for some node $n$ in $\mathbf{b}$. Conversely, if $C_n \Vdash \phi(\alpha_0, ..., \alpha_s, ...)$ for some node $n$ in $\mathbf{b}$, by definition of the forcing there is a node $m$ above $n$ in $\mathbf{b}$ and a function $f_{nm}: D_n \to D_m$ for which $C_m \Vdash \theta_i(f_{nm}(\alpha_0), ..., f_{nm}(\alpha_s), ...)$ for some $i<\gamma$, so that by inductive hypothesis we get $\mathbf{D_b} \vDash \phi(\overline{\alpha_0}, ..., \overline{\alpha_s}, ...)$.
  
  \item Finally, if $\phi$ is of the form $\exists \mathbf{x} \psi(\mathbf{x}, x_0, ..., x_s, ...)$ and $\mathbf{D_b} \vDash \phi(\overline{\alpha_0}, ..., \overline{\alpha_s}, ...)$, then $\mathbf{D_b} \vDash \psi(\boldsymbol{\overline{\alpha}}, \overline{\alpha_0}, ..., \overline{\alpha_s}, ...)$ for some $\boldsymbol{\overline{\alpha}}$, and then $C_n \Vdash \psi(\boldsymbol{\alpha}, \alpha_0, ..., \alpha_s, ...)$ for some node $n$ by inductive hypothesis. Conversely, if $C_n \Vdash \phi(\alpha_0, ..., \alpha_s, ...)$ for some node $n$ in $\mathbf{b}$, then by definition of the forcing there is a node $m$ above $n$ in $\mathbf{b}$ and a function $f_{nm}: D_n \to D_m$ for which $C_m \Vdash \psi(f_{nm}(\boldsymbol{\alpha}), f_{nm}(\alpha_0), ..., f_{nm}(\alpha_s), ...)$, which implies that $\mathbf{D_b} \vDash \psi(\boldsymbol{\overline{\alpha}}, \overline{\alpha_0}, ..., \overline{\alpha_s}, ...)$ and hence $\mathbf{D_b} \vDash \phi(\overline{\alpha_0}, ..., \overline{\alpha_s}, ...)$. 
 \end{enumerate}

 Since $\psi(\overline{\boldsymbol{\alpha}})$ is satisfied in all $\kappa$-coherent models of the theory satisfying $\phi(\overline{\boldsymbol{\alpha}})$, it is satisfied in all models of the form $\mathbf{D_b}$ (even if the structure $\mathbf{D_b}$ is exploding). Hence, $\psi(\boldsymbol{\alpha})$ is forced at a certain node of every branch of the tree. By taking limits over the diagram formed by each branch $\mathbf{b}$ we get nodes at level $\delta$ which also have to force $\psi(\boldsymbol{\alpha})$. Because these nodes form a basic covering family, $\psi(\boldsymbol{\alpha})$ is therefore forced at the root, as we wanted to prove. 
\end{proof}

One can remove the restriction on the cardinality of the theory if one assumes instead that $\kappa$ is a weakly (resp. strongly) compact cardinal:

\begin{proposition}\label{scohcomp}
If $\kappa$ is a weakly (resp. strongly) compact cardinal, $\kappa$-coherent theories of cardinality at most $\kappa$ (resp. of arbitrary cardinality) are complete with respect to \Sets-valued models.
\end{proposition}

\begin{proof}
 Suppose that the sequent $\phi \vdash_{\mathbf{x}} \psi$ is valid in every model of a certain theory but not provable. Then it is not provable in any subtheory of cardinality less than $\kappa$. Therefore, if we add to the language new constants $\mathbf{c}$ and axioms $\top \vdash \phi(\mathbf{c})$ and $\psi(\mathbf{c}) \vdash \bot$, any subtheory of cardinality less than $\kappa$ together with these two new axioms has, by Proposition \ref{cohcomp}, a model. Since $\kappa$ is weakly (resp. strongly) compact, the whole theory has a model, which provides a model for the original theory where $\phi \vdash_{\mathbf{x}} \psi$ is not valid.
\end{proof}

\begin{rmk}
 In the particular case when $\kappa=\omega$, Proposition \ref{scohcomp} reduces to the completeness theorem for coherent logic. The proof given is thus an alternative path to the usual categorical treatment of Henkinization that we see in \cite{johnstone}, D 1.5, and is closer to the methods in \cite{clr}.
\end{rmk}

Incidentally, a slight modification on the definition of witnessing covers that build the underlying tree of the partial Beth model in Theorem \ref{shcomp} yields a full model for $\kappa$-first-order theories:

\begin{thm}
 Let $\kappa$ be an inaccessible cardinal. Then any $\kappa$-first-order theory of cardinality at most $\kappa$ has a universal\footnote{By universal we mean a conservative model, that is, one in which true sequents are provable.} Beth model.
\end{thm}

\begin{proof}
 We consider, for any object, a well-ordering of the set of basic covering families over the object, which are given by jointly covering sets of arrows (of cardinality less than $\kappa$), that is, we include in the set $S$ all possible covering families and not just the ones that witness some disjunction or existential formula. Since $\kappa$ is inaccessible, it is easy to see that this set has cardinality $\kappa$. We define a tree of height $\kappa$ as before, using a well-ordering $f: \kappa \times \kappa \to \kappa$ as in Lemma \ref{dwo}. There is an obvious interpretation of the function symbols in the subset underlying each node, as before. If for relations $R$ (including equality), we set by definition $R_q(\mathbf{s}(\boldsymbol{\alpha}))$ if and only if $q$ forces $R(\mathbf{s}(\boldsymbol{\alpha}))$ in the sheaf semantics of the topos (we identify the category with its image through Yoneda embedding), we will have, as before, the:
 
 $Claim:$ For every node $p$, every tuple $\boldsymbol{\alpha}$ and every formula $\phi \in S$, $p \Vdash \phi(\boldsymbol{\alpha})$ if and only if $p \Vdash_B \phi(\boldsymbol{\alpha})$.\\

The proof is again by induction on $\phi$, and the steps for the atomic formulas, conjunctions, disjunctions and existential quantification are the same as before. We need only to consider implication and universal quantification.

\begin{enumerate}
  \item Suppose $\phi=\psi \to \theta$. If $p \Vdash \psi(\boldsymbol{\alpha}) \to \theta(\boldsymbol{\alpha})$, for every $f: c \to p$ in the category one has $c \Vdash \psi(\boldsymbol{\alpha} f) \implies c \Vdash \theta(\boldsymbol{\alpha} f)$. In particular, this holds when $c$ is any node $q$ in the tree above $p$, and by inductive hypothesis one has $q \Vdash_B \psi(\boldsymbol{\alpha} f) \implies q \Vdash_B \theta(\boldsymbol{\alpha} f)$ for all such nodes. Therefore, $p \Vdash_B \psi(\boldsymbol{\alpha}) \to \theta(\boldsymbol{\alpha})$.
  
  Conversely, suppose that $p \Vdash_B \psi(\boldsymbol{\alpha}) \to \theta(\boldsymbol{\alpha})$ and consider an arrow $f: c \to p$. Together with the identity, this arrow forms a covering family which appears at some point in the well-ordering and is hence pulled back along paths $g_j$ of a subtree to build the next level of the subtree over $p$. Suppose that $c \Vdash \psi(\boldsymbol{\alpha})$; then $g_j^*(c) \Vdash \psi(\boldsymbol{\alpha} g'_j)$, so by inductive hypothesis one has $g_j^*(c) \Vdash_B \psi(\boldsymbol{\alpha} g'_j)$. Therefore, we get $g_j^*(c) \Vdash_B \theta(\boldsymbol{\alpha} g'_j)$, and using once more the inductive hypothesis, $g_j^*(c) \Vdash \theta(\boldsymbol{\alpha} g'_j)$. But $g'_j=f^*(g_j): g_j^*(c) \to c$ is a basic cover of $c$ (since the $g_j$ form a basic cover of $p$), and hence we will have $c \Vdash \theta(\boldsymbol{\alpha})$. We have, thus, proved that $p \Vdash \psi(\boldsymbol{\alpha}) \to \theta(\boldsymbol{\alpha})$
  
  \item Suppose $\phi=\forall \mathbf{x} \psi(\mathbf{x}, \boldsymbol{\alpha})$. If $p \Vdash \forall \mathbf{x} \psi(\mathbf{x}, \boldsymbol{\alpha})$, for every $f: c \to p$ in the category and every $\boldsymbol{\beta}: c \to [\mathbf{x}, \top]$ one has $c \Vdash \psi(\boldsymbol{\beta}, \boldsymbol{\alpha})$. In particular, this holds when $c$ is any node $q$ in the tree above $p$, and by inductive hypothesis one has $q \Vdash_B \psi(\boldsymbol{\beta}, \boldsymbol{\alpha})$ for all such nodes. Therefore, $p \Vdash_B \forall \mathbf{x} \psi(\mathbf{x}, \boldsymbol{\alpha})$.
  
  Conversely, suppose that $p \Vdash_B \forall \mathbf{x} \psi(\mathbf{x}, \boldsymbol{\alpha})$ and consider an arrow $f: c \to p$. Together with the identity, this arrow forms a covering family which appears at some point in the well-ordering and is hence pulled back along the paths $g_j$ of a subtree to build the next level of the subtree over $p$. Suppose we have some $\boldsymbol{\beta}: c \to [\mathbf{x}, \top]$; then we have arrows $\boldsymbol{\beta} f^*(g_j): g_j^*(c) \to [\mathbf{x}, \top]$, and by definition we must have $g_j^*(c) \Vdash_B \psi(\boldsymbol{\beta} f^*(g_j), \boldsymbol{\alpha} fg'_j)$, so by inductive hypothesis one has $g_j^*(c) \Vdash \psi(\boldsymbol{\beta} f^*(g_j), \boldsymbol{\alpha} fg'_j)$. But $f^*(g_j): g_j^*(c) \to c$ is a basic cover of $c$ (since the $g_j$ form a basic cover of $p$), and hence we will have $c \Vdash \psi(\boldsymbol{\beta}, \boldsymbol{\alpha})$. We have thus proved that $p \Vdash \forall \mathbf{x} \psi(\mathbf{x}, \boldsymbol{\alpha})$.
 \end{enumerate}

This finishes the proof.
\end{proof}

We conclude with an equivalent characterization of weak (resp. strong) compactness.

\begin{defs}
A $\kappa$-complete lattice will be called $\kappa$-distributive if it satisfies the intuitionistic distributivity law, i.e., if $a \wedge \bigvee_{i<\gamma} b_{i} \vdash \bigvee_{i<\gamma}a \wedge b_{i}$ for every $\gamma<\kappa$, and if the following propositional version of transfinite transitivity property holds:   for every $\gamma<\kappa$ and all elements $\{a_f: f \in \gamma^{\beta}, \beta<\gamma\}$  such that 

$$a_{f} \leq \bigvee_{g \in \gamma^{\beta+1}, g|_{\beta}=f} a_{g}$$
\noindent for all $f \in \gamma^{\beta}, \beta<\gamma$, and 

$$a_{f} = \bigwedge_{\alpha<\beta}a_{f|_{\alpha}}$$
\noindent for all limit $\beta$, $f \in \gamma^{\beta}, \beta<\gamma$, we have that 

$$a_{\emptyset} \leq \bigvee_{f \in \gamma^{\gamma}} \bigwedge_{\beta<\gamma}a_{f|_{\beta}}.$$

A $\kappa$-complete filter in the lattice is a filter $\mathcal{F}$ such that whenever $a_i \in \mathcal{F}$ for every $i \in I$, $|I|<\kappa$, then $\bigwedge_{i \in I}a_i \in \mathcal{F}$. Dually, a $\kappa$-complete ideal in the lattice is an ideal $\mathcal{I}$ such that whenever $a_i \in \mathcal{I}$ for every $i \in I$, $|I|<\kappa$, then $\bigvee_{i \in I}a_i \in \mathcal{I}$. A $\kappa$-prime filter in the lattice is a filter $\mathcal{F}$ such that whenever $\bigvee_{i \in I}a_i$ is in $\mathcal{F}$ for $|I|<\kappa$ then $a_i \in \mathcal{F}$ for some $i \in I$.
\end{defs}

\begin{lemma}\label{quotient}
 Given any proper $\kappa$-complete filter $\mathcal{F}$ in a $\kappa$-complete, $\kappa$-distributive lattice $\mathcal{L}$, there exists a surjective morphism $\theta: \mathcal{L} \to \mathcal{K}$ into a $\kappa$-complete, $\kappa$-distributive lattice $\mathcal{K}$ such that $\theta^{-1}(1)=\mathcal{F}$.
\end{lemma} 

\begin{proof}
The proof is an adaptation of that exposed e.g. in \cite{maclane-moerdijk}, V.9,  for Heyting algebras. Consider first the case in which the filter $\mathcal{F}$ is principal, given by some element $u$ as the set $\{u': u' \geq u\}$. Define now an equivalence relation on $\mathcal{L}$ by setting $a \cong b \iff a \wedge u=b \wedge u$. Denoting by $a_u$ the equivalence class of $a$, we can give the quotient $\mathcal{L}/u$ a structure of a $\kappa$-complete lattice with the partial order given by $a_u \leq b_u \iff a \wedge u \leq b \wedge u$. On the resulting poset of equivalence classes, the meet and the join are given as follows:

$$\bigwedge_{i \in \gamma}a^i_u =\left(\bigwedge_{i \in \gamma}a^i\right)_u$$
$$\bigvee_{i \in \gamma}a^i_u =\left(\bigvee_{i \in \gamma}a^i\right)_u$$
\\
That these operations are well defined follows from the fact that $\mathcal{L}$ is $\kappa$-distributive. With these operations, it also follows that $\mathcal{L}/u$ inherits the $\kappa$-distributivity.  

Now, if we had $u \leq v \leq w$ in $\mathcal{L}$, there are evident morphisms $\mathcal{L}/w \to \mathcal{L}/v \to \mathcal{L}/u$ with composite $\mathcal{L}/w \to \mathcal{L}/u$. Then for a general filter $\mathcal{F}$, we just define $\mathcal{K}$ as the colimit:

$$\lim_{u \in \mathcal{F}}\mathcal{L}/u.$$
\\
Since $\mathcal{F}$ is $\kappa$-complete, it follows that the colimit is $\kappa$-filtered, and hence it inherits all the structure of $\kappa$-complete, $\kappa$-distributive lattice of each $\mathcal{L}/u$. This finishes the proof.
\end{proof}

We can state now:
 
\begin{proposition}\label{filter}
 A cardinal $\kappa$ is weakly (resp. strongly) compact if and only if, given any proper $\kappa$-complete ideal $\mathcal{I}$ in a $\kappa$-complete, $\kappa$-distributive lattice $\mathcal{L}$ of cardinality at most $\kappa$ (resp. of arbitrary cardinality), and given a $\kappa$-complete filter $\mathcal{F}$ disjoint from $\mathcal{I}$, there exists a $\kappa$-complete, $\kappa$-prime filter containing $\mathcal{F}$ and disjoint from $\mathcal{I}$.
\end{proposition}

\begin{proof}
 The ``if'' part follows by restricting ourselves to the case when the lattice is a Boolean algebra; we prove here the ``only if'' part. Consider the quotient $\mathcal{L}'$ of the lattice $\mathcal{L}$ by $\mathcal{F}$ given by Lemma \ref{quotient}; this is a $\kappa$-complete and $\kappa$-distributive lattice, and the image of $\mathcal{I}$ is a proper $\kappa$-complete ideal $\mathcal{I}'$. Define a new propositional variable $P_a$ for each element $a$ of $\mathcal{L}'$ and consider the theory axiomatized by the following axioms:
 
 \begin{enumerate}
  \item $P_a \vdash \bot$ for all $a \in \mathcal{I}'$
  \item $\bigwedge_{i \in I}P_{a_i} \vdash P_{\left(\bigwedge_{i \in I}a_i\right)}$ for all $a_i$ and $|I|<\kappa$
  \item $P_{\left(\bigvee_{i \in I}a_i\right)} \vdash \bigvee_{i \in I}P_{a_i}$ for all $a_i$ and $|I|<\kappa$
 \end{enumerate}

This is a theory over the $\kappa$-coherent fragment which has cardinality at most $\kappa$. Each subtheory of cardinality less than $\kappa$ involves $\gamma<\kappa$ propositional variables, whose correspondent elements generate a (non trivial) $\kappa$-complete sublattice of $\mathcal{L}'$ of cardinality $2^{\gamma}<\kappa$ (as every element there is equivalent to one of the form $\bigvee_{i<\gamma} \bigwedge_{j<\gamma} P_{a_{ij}}$) containing, thus, less than $\kappa$ many $\kappa$-prime filters. By Proposition \ref{cohcomp}, the intersection of all $\kappa$-prime filters of any such sublattice is $\{1\}$, and hence any such sublattice contains at least one $\kappa$-prime filter disjoint from elements of $\mathcal{I}'$. This shows that each subtheory of cardinality less than $\kappa$ has a model. Since $\kappa$ is weakly (resp. strongly) compact, the whole theory has a model, which corresponds to a $\kappa$-complete, $\kappa$-prime filter of $\mathcal{L}'$ whose preimage along the quotient map provides a $\kappa$-complete, $\kappa$-prime filter in $\mathcal{L}$ containing $\mathcal{F}$ and disjoint from $\mathcal{I}$.
\end{proof}

\begin{rmk}
 Proposition \ref{filter} can be used to provide a representation theorem for $\kappa$-complete, $\kappa$-distributive lattices of cardinality at most $\kappa$ (resp. of arbitrary cardinality) in terms of lattices of subsets. More especifically, the $\kappa$-complete lattice morphism $f$ that maps an element $e \in \mathcal{L}$ to the set of $\kappa$-complete, $\kappa$-prime filters containing $e$ provides an isomorphism of lattices $\mathcal{L} \to f(\mathcal{L})$ precisely when $\kappa$ is weakly (resp. strongly) compact. 
\end{rmk}

\subsection{Completeness of infinitary intuitionistic first-order logic}

Having now at hand a completeness theorem for $\kappa$-coherent theories, we can adapt the proof of Joyal's theorem by replacing the category of coherent models with that of $\kappa$-coherent models. As a result, we get:

\begin{thm}\label{bt}
 If $\kappa$ is weakly (resp. strongly) compact, $\kappa$-first-order theories of cardinality at most $\kappa$ (resp. of arbitrary cardinality) are complete with respect to Kripke models.
\end{thm}

\begin{proof}
Consider the syntactic category $\mathcal{C}$ of the $\kappa$-coherent Morleyization $\theory^m$ of the theory $\theory$. Let $\mathcal{C}oh(\mathcal{C})$ be the category of $\kappa$-coherent models of $\theory^m$ of size at most $\kappa$ (resp. of arbitrary size), and where arrows are model homomorphisms. We have a functor $ev: \mathcal{C} \to \Sets^{\mathcal{C}oh(\mathcal{C})}$ sending an object $A$ to the evaluation functor $ev(A)$. It is clear that this functor is $\kappa$-coherent, and by Proposition \ref{scohcomp}, it is also conservative (this is because any model contains an elementary submodel of cardinality at most $\kappa$, by the usual L\"owenheim-Skolem construction). We must prove that $ev$ also preserves $\forall$. 

Given an arrow $f: A \to B$, a subobject $C \rightarrowtail A$ and the subobject $Y=\forall_f(C) \rightarrowtail B$, we need to show that $ev(Y)=\forall_{ev(f)} (ev(C))$ as subobject of $ev(B)$. By the definition of $\forall$ in the Heyting category $\Sets^{\mathcal{C}oh(\mathcal{C})}$, this reduces to proving the following equivalence, for every $\mathbf{y} \in ev(B)(M)=M(B)$:

$$\mathbf{y} \in ev(Y)(M) \iff \text{ For every model } N, \text{ for every model homomorphism}$$
                            $$\phi: M \to N,$$ 

$$(ev(f)_N)^{-1}(\phi_B(\mathbf{y})) \subseteq ev(C)(N)$$

that is:

$$\mathbf{y} \in M(Y) \iff \text{ For every model } N, \text{ for every model homomorphism}$$
                         $$\phi: M \to N,$$

$$(N(f))^{-1}(\phi_B(\mathbf{y})) \subseteq N(C)$$

The implication $\implies$ can be proven as follows: if $\mathbf{y} \in M(Y)$, then $\phi_B(\mathbf{y}) \in N(Y)$, and so, since $N$ is $\kappa$-coherent, $\phi_B(\mathbf{y}) = N(f)(\mathbf{x})$ gives $\mathbf{x} \in N(f)^{-1} (N(\forall_f(C)))=N(f^{-1} \forall_f(C)) \subseteq N(C)$.

Let us focus on the other implication. Consider the following diagram in $\mathcal{C}$:

\begin{displaymath}
\xymatrix{
C=[\mathbf{x}, \theta] \ar@{ >->}[dd] & & \forall_f(C)=[\mathbf{y}, \gamma] \ar@{ >->}[dd]\\
 & & \\
A=[\mathbf{x}, \phi]  \ar@{->}[rr]_{f=[\mathbf{x}\mathbf{y}, \lambda]} & & B=[\mathbf{y}, \psi] \\
}
\end{displaymath}

Applying the functor $ev$ and evaluating at a model $M$ gives the diagram:

\begin{displaymath}
\xymatrix{
\{\mathbf{d} | M \vDash \theta(\mathbf{d})\} \ar@{ >->}[dd] & & \{\mathbf{c} | M \vDash \gamma(\mathbf{c})\} \ar@{ >->}[dd]\\
 & & \\
\{\mathbf{d} | M \vDash \phi(\mathbf{d})\}  \ar@{->}[rr]_{\{\mathbf{d}, \mathbf{c} | M \vDash \lambda(\mathbf{d}, \mathbf{c})\}} & & \{\mathbf{c} | M \vDash \psi(\mathbf{c})\} \\
}
\end{displaymath}

Given $\mathbf{c} \in \forall_{ev(f)} (ev(C))$, we need to prove that $M \vDash \gamma(\mathbf{c})$. Consider the positive diagram of $M$, $Diag_+(M)$, which, in a language extended with constants $c$ for every element $c$ of the underlying set of $M$, consists of all sequents of the form $\top \vdash \psi(c_0, ..., c_{\alpha}, ...)$ for every positive atomic $\psi$ such that $M \vDash \psi(c_0, ..., c_{\alpha}, ...)$ (we identify the constants symbols with the elements of $M$, to simplify the exposition). If $N'$ is a model of $Th(M)$ of size at most $\kappa$ (resp. of arbitrary size), then, defining $N$ as the reduct of $N'$ with respect to the elements $\{c^{N'}: c \in M\}$ we can define $\phi: M \to N$ by $\phi(c)=c^{N'}$, which is a well defined model homomorphism. But we know that for all $\phi: M \to N$ one has $N(f)^{-1}(\phi_B(\mathbf{c})) \subseteq N(C)$. This implies that for all models $N'$ of $Th(M)$ of size at most $\kappa$ (resp. of arbitrary size), the sequent $\lambda(\mathbf{x}, \mathbf{c}/\mathbf{y}) \vdash_{\mathbf{x}} \theta(\mathbf{x})$ holds, and therefore, the sequent $\psi(\mathbf{c}) \wedge \lambda(\mathbf{x}, \mathbf{c}/\mathbf{y}) \vdash_{\mathbf{x}} \theta(\mathbf{x})$ also holds.

By Proposition \ref{scohcomp}, this means that such a sequent is provable in $Th(M)$. Besides sequents in $\theory^m$, this proof uses less than $\kappa$ sequents of the general form $\top \vdash \phi_i(\mathbf{c}, \mathbf{c_0}, ..., \mathbf{c_{\alpha}}, ...)$, where the $\phi_i$ are positive atomic sentences corresponding to the diagram of $M$ and the $\mathbf{c_i}$ are elements of $M$. Considering the conjunction $\xi$ of the $\phi_i$, we see that there is a proof in $\theory^m$ from:

$$\top \vdash \xi(\mathbf{c}, \mathbf{c_0}, ..., \mathbf{c_{\alpha}}, ...)$$
\\
to 

$$\psi(\mathbf{c}) \wedge \lambda(\mathbf{x}, \mathbf{c}/\mathbf{y}) \vdash_{\mathbf{x}} \theta(\mathbf{x})$$
\\
By the deduction theorem (Lemma \ref{dt}), since $\xi(\mathbf{c}, \mathbf{c_0}, ..., \mathbf{c_{\alpha}}, ...)$ is a sentence, we obtain in $\theory^m$ a derivation of:

$$\xi(\mathbf{c}, \mathbf{c_0}, ..., \mathbf{c_{\alpha}}, ...) \wedge \psi(\mathbf{c}) \wedge \lambda(\mathbf{x}, \mathbf{c}/\mathbf{y}) \vdash_{\mathbf{x}} \theta(\mathbf{x})$$
\\
But it is always possible to replace the constants by variables as long as they are added to the contexts of the sequents, so using the existential rule, we have also a derivation of:

$$\exists \mathbf{x_0} ... \mathbf{x_{\alpha}} ... \xi(\mathbf{y}, \mathbf{x_0}, ..., \mathbf{x_{\alpha}}, ...) \wedge \psi(\mathbf{y}) \wedge \lambda(\mathbf{x}, \mathbf{y}) \vdash_{\mathbf{x} \mathbf{y}} \theta(\mathbf{x})$$
\\
Calling $Y'=[\mathbf{y}, \Phi(\mathbf{y})]$ the subobject of $B$ given by the interpretation in $\mathcal{C}$ of the formula:

$$\exists \mathbf{x_0} ... \mathbf{x_{\alpha}} ... \xi(\mathbf{y}, \mathbf{x_0}, ..., \mathbf{x_{\alpha}}, ...) \wedge \psi(\mathbf{y})$$
\\
we have a proof of the sequent:

$$\Phi(\mathbf{y}) \wedge \lambda(\mathbf{x}, \mathbf{y}) \vdash_{\mathbf{x} \mathbf{y}} \theta(\mathbf{x})$$
\\
and hence also of the sequent:

$$\exists \mathbf{y} (\Phi(\mathbf{y}) \wedge \lambda(\mathbf{x}, \mathbf{y})) \vdash_{\mathbf{x}} \theta(\mathbf{x})$$
\\
Now the antecedent is precisely the pullback of the subobject $\Phi(\mathbf{y})$ of $B$ along $f$, so by adjunction we have $Y' \leq \forall_f(C)=[\mathbf{y}, \gamma]$, i.e., the sequent $\Phi(\mathbf{y}) \vdash_{\mathbf{y}} \gamma(\mathbf{y})$ is provable. Therefore, since $M \vDash \Phi(\mathbf{c})$, it follows that $M \vDash \gamma(\mathbf{c})$, as we wanted to prove.

\end{proof}

There is also the following converse of Theorem \ref{bt}:

\begin{proposition}\label{converse}
The completeness theorem of $\kappa$-first order/$\kappa$-coherent theories of cardinality $\kappa$ (resp. of arbitrary cardinality) with respect to Kripke models/Tarski models implies that $\kappa$ is weakly (resp. strongly) compact.
\end{proposition}

\begin{proof}
We prove first that completeness of $\kappa$-coherent theories of cardinality $\kappa$ is entailed by Kripke completeness of $\kappa$-first-order theories of cardinality $\kappa$. Given a $\kappa$-coherent theory, suppose a coherent sequent is valid in all $\kappa$-coherent models in \Sets. Then it is necessarily forced at every node of every Kripke model, and therefore provable from the axioms of the theory in $\kappa$-first-order logic. Because $\kappa$-first-order logic is conservative over $\kappa$-coherent logic (as there is, by Lemma \ref{shemb}, a conservative embedding of the syntactic category of the latter into a sheaf topos verifying the former), it has to be provable already in $\kappa$-coherent logic.

To prove that the completeness of $\kappa$-coherent theories of cardinality $\kappa$ implies weak compactness, given a tree of height $\kappa$ and levels of size less than $\kappa$, consider the theory of a cofinal branch, over a language containing a unary relation symbol $P$ and one constant $a$ for every node in the tree and axiomatized as follows:
 
 $$\top \vdash \bigvee_{a \in L_{\alpha}}P(a)$$
 \\
 for each $\alpha<\kappa$, where $L_{\alpha}$ is the level of height $\alpha$;
 
 $$P(a) \wedge P(b) \vdash \bot$$
 \\
 for each pair $a \neq b \in L_{\alpha}$ and each $\alpha<\kappa$;
 
 $$P(a) \vdash P(b)$$
 \\
 for each pair $a, b$ such that $a$ is a successor of $b$.

Then the theory is certainly consistent within $\mathcal{L}_{\kappa, \kappa}$, as every subtheory of cardinality less than $\kappa$ has a Tarski model, so by completeness it follows that the whole theory has a Tarski model, corresponding to a cofinal branch.
 
Finally, in case we have Kripke completeness of $\kappa$-first-order theories of arbitrary cardinality, we can deduce in a similar way as before the completeness of $\kappa$-coherent theories of arbitrary cardinalities. To show that this latter implies strong compactness, consider a $\kappa$-complete filter $\mathcal{F}$ in the lattice $\mathcal{L}$ of subsets of a set. In a language containing a constant $a$ for every $a \in \mathcal{L}/\mathcal{F}$ and a unary relation symbol $P$, consider the following theory of a $\kappa$-complete ultrafilter:
 
 \begin{enumerate}
  \item $P(a) \vdash P(b)$ for every pair $a \leq b$ in $\mathcal{L}$
  \item $\bigwedge_{i<\gamma}P(a_i) \vdash P(\bigwedge_{i<\gamma}a_i)$ for all families $\{a_i\}_{i<\gamma}$ such that $\gamma<\kappa$
  \item $\top \vdash P(a) \vee P(\neg a)$ for every $a \in \mathcal{L}$
 \end{enumerate}
 
 Since the theory is consistent (as it has a model in $\mathcal{L}/\mathcal{F}$ itself), by $\kappa$-coherent completeness it has a Tarski model, which provides a $\kappa$-complete ultrafilter in $\mathcal{L}/\mathcal{F}$ whose preimage along the quotient map yields a $\kappa$-complete ultrafilter in $\mathcal{L}$ extending $\mathcal{F}$. Therefore, $\kappa$ is strongly compact.
\end{proof}

\subsection{Heyting cardinals}

It is known, as we will prove in the next section, that for inaccessible $\kappa$, $\kappa$-first-order classical logic is complete for theories of cardinality less than $\kappa$. This was first observed by Karp in \cite{karp}, and one naturally wonders if the analogous situation holds in the intuitionistic case. This motivates the following:

\begin{defs}
 We say that $\kappa$ is a \emph{Heyting} cardinal if $\kappa$ is inaccessible and $\kappa$-first-order logic is complete for theories of cardinality less than $\kappa$.
\end{defs}

By Theorem \ref{bt} we know that, in terms of strength, a Heyting cardinal lies between inaccessible and weakly compact cardinals. Its exact large cardinal strength and relationship with other classical large cardinal properties is however currently unknown, although we will outline some evidence that relates it to an instance of weak compactness.

As applications of completeness, we obtain the disjunction and existence properties:

\begin{proposition}\label{dp}
If $\kappa$ is a Heyting cardinal, $\kappa$-first-order intuitionistic logic $\mathcal{L}_{\kappa, \kappa}$ has the infinitary disjunction property. That is, if $\top \vdash \bigvee_{i \in I}\phi_i$ is provable in the empty theory, then, for some $i \in I$, $\top \vdash \phi_i$ is already provable.
\end{proposition}

\begin{proof}
 This is a straightforward generalization of the usual semantic proof in the finitary case, based on the completeness with respect to Kripke models over trees. If no sequent $\top \vdash \phi_i$ was provable, there would be a countermodel for each. Then we can build a new Kripke tree appending to these countermodels a bottom node whose underlying set consists just of the constants of the language, with the obvious injections into the roots of the countermodels (and forcing no atoms). Such a Kripke tree would then be a countermodel for $\top \vdash \bigvee_{i \in I}\phi_i$.
\end{proof}

\begin{proposition}
If $\kappa$ is a Heyting cardinal, $\kappa$-first-order intuitionistic logic $\mathcal{L}_{\kappa, \kappa}$ over a language without function symbols and with at least one constant symbol has the infinitary existence property. That is, if $\top \vdash \exists \mathbf{x} \phi(\mathbf{x})$ is provable in the empty theory, then, for some constants $\mathbf{c}$, $\top \vdash \phi(\mathbf{c})$ is already provable.
\end{proposition}

\begin{proof}
 The proof of this is similar to the one given for the disjunction property. If no sequent $\top \vdash \phi(\mathbf{c})$ was provable, there would be a countermodel for each choice of $\mathbf{c}$. Then we can build a new Kripke tree appending to these countermodels a bottom node forcing no atoms, whose underlying domain contains just the constants of the language, again with the obvious injections into the roots of the countermodels (and forcing no atoms). Such a Kripke tree would then be a countermodel for $\top \vdash \exists \mathbf{x} \phi(\mathbf{x})$.
\end{proof}

It is possible that the disjunction and existence properties themselves represent a large cardinal notion different from mere inaccessibility. To see this, note first that the methods of \cite{jongh} allow to prove that $\kappa$-propositional logic (including the distributivity rule) with at least two atoms has $\kappa$ many mutually non-equivalent formulas\footnote{This is a remarkable property of infinitary intuitionistic logic. In classical infinitary propositional logic $\mathcal{L}_{\kappa}$, for example, the distributivity property implies, by a theorem of Tarski (see \cite{tarski}), that when there are less than $\kappa$-many atoms, the logic has less than $\kappa$ many mutually non-equivalent formulas.}, and so the same holds in general for $\kappa$-first-order logic. This implies that the $\kappa$-coherent Morleyization of $\kappa$-first-order logic contains $\kappa$ many relation symbols which are mutually non-equivalent, and so it has has signature of cardinality $\kappa$ (as well as $\kappa$-many axioms). By Proposition \ref{cohcomp}, this theory is $\kappa$-satisfiable, i.e., any subset of less than $\kappa$ axioms has a model. But the following proposition will show that the whole theory has a model, without using an instance of weak compactness:

\begin{proposition}
 If $\kappa$-first-order logic has the disjunction and the existence properties, its $\kappa$-coherent Morleyization has a \Sets-valued model.
\end{proposition}

\begin{proof}
 Consider the syntactic category $\mathcal{C}$ of the theory $\theory$. We will show that the representable functor $y(1)=[1, -]: \mathcal{C} \to \Sets$ is a $\kappa$-coherent functor, providing thus the model of the Morleyized theory $\theory^m$. Since $y(1)$ preserves all limits, it will preserve in particular $\kappa$-limits, so it is enough to show that it preserves $\kappa$-unions and covers.
 
 That $\bigvee_{i<\gamma} y(1)(A_i) \leq y(1)(\bigvee_{i<\gamma}A_i)$ follows since $y(1)$ preserves limits and hence monomorphisms. To see the other inequality, let $f: 1 \to \bigvee_{i<\gamma}A_i$. Then we have $1=\bigvee_{i<\gamma}f^{-1}(A_i)$. By the disjunction property, $1=f^{-1}(A_j)$ for some $j<\gamma$, and hence $f$ factors through $A_j \rightarrowtail \bigvee_{i<\gamma}A_i$. This shows that $y(1)(\bigvee_{i<\gamma}A_i) \leq \bigvee_{i<\gamma} y(1)(A_i)$, as we wanted.
 
 Finally, suppose we have a cover $g: A \twoheadrightarrow B$; we must show that any $h: 1 \to B$ factors through $A$. For this, in turn, it is enough to take the pullback $h^{-1}(g): h^{-1}(A) \twoheadrightarrow 1$ and find a section $s$ for $h^{-1}(g)$, then we can take the required factorization as the composite $g^{-1}(h) \circ s: 1 \to A$. Let $h^{-1}(A)=[\mathbf{x}, \phi(\mathbf{x})]$. Since $h^{-1}(g)$ is a cover, we have $\top \vdash \exists \mathbf{x} \phi(\mathbf{x})$. By the existence property, there are some $\mathbf{c}$ with $[\![\mathbf{c}]\!]: 1 \to [\mathbf{x}, \top]$ such that $\top \vdash \phi(\mathbf{c})$. But $\phi(\mathbf{c})$ is the pullback of $[\mathbf{x}, \phi(\mathbf{x})] \rightarrowtail [\mathbf{x}, \top]$ along $[\![\mathbf{c}]\!]$, which provides the required section since $1=\phi(\mathbf{c})$. This finishes the proof. 
\end{proof}

To summarize, if the disjunction and existence properties hold for $\kappa$ (for example if $\kappa$ is a Heyting cardinal), there is a $\kappa$-satisfiable $\mathcal{L}_{\kappa, \kappa}$-theory over a signature of cardinality $\kappa$, and with $\kappa$-many axioms, which has a model. Since there are $\kappa$ many mutually non-equivalent relation symbols, it follows that models for subtheories of less than $\kappa$ many axioms do not trivially extend to models of the whole theory. The existence of a model for the whole theory could thus require an instance of weak compactness in an essential way.

\subsection{Karp's theorem and completeness of classical infinitary systems}

In the classical case, when the syntactic category is Boolean, we can prove that Proposition \ref{cohcomp} reduces to the completeness theorem of Karp in \cite{karp}, since we know that the transfinite transitivity property can be rewritten into the classical distributivity and dependent choice axiom schemata. This reduction will be possible due to classical Morleyization, which is the infinitary version of the process explained in \cite{johnstone}, D 1.5:

\begin{thm}\label{karp}
 (Karp) If $\kappa$ is inaccessible, theories of cardinality less than $\kappa$ within the classical first-order system of Karp are complete with respect to \Sets-valued models.
\end{thm}

\begin{proof}
 Suppose the sequent $\phi \vdash_{\alg{x}} \psi$ is valid in all models in \Sets. Let $S$ be the set of subformulas of some axiom of $\theory$ or of $\phi$ or $\psi$. We consider now the classical Morleyized theory $\theory^m$, over a signature $\Sigma^m$ that extends the original signature $\Sigma$ by adding for each $\kappa$-first-order formula $\phi \in S$ over $\Sigma$ with free variables \alg{x} two new relation symbols $C_{\phi}(\alg{x})$ and $D_{\phi}(\alg{x})$, and whose axioms are:

\begin{enumerate}[(i)]

\item $C_{\phi}(\alg{x}) \wedge D_{\phi}(\alg{x}) \vdash_{\alg{x}} \bot$;
\item $\top \vdash_{\alg{x}} C_{\phi}(\alg{x}) \vee D_{\phi}(\alg{x})$;
\item $C_{\phi}\dashv\vdash_{\alg{x}}\phi$ for every atomic formula $\phi$;
\item $C_{\phi}\vdash_{\alg{x}}C_{\psi}$ for every axiom $\phi\vdash_{\alg{x}}{\psi}$ of \theory;
\item $C_{\bigwedge_{i<\gamma}\phi_i}\dashv\vdash_{\alg{x}}\bigwedge_{i<\gamma}C_{\phi_i}$;
\item $C_{\bigvee_{i<\gamma}\phi_i}\dashv\vdash_{\alg{x}}\bigvee_{i<\gamma}C_{\phi_i}$;
\item $C_{\phi \rightarrow \psi}\dashv\vdash_{\alg{x}}D_{\phi} \vee C_{\psi}$
\item $C_{\exists{\mathbf{y}}\phi}\dashv\vdash_{\alg{x}}\exists{\mathbf{y}}C_{\phi}$;
\item $D_{\forall{\mathbf{y}}\phi}\dashv\vdash_{\alg{x}}\exists{\mathbf{y}}D_{\phi}$.

\end{enumerate} 

These axioms are all $\kappa$-coherent and they ensure that the interpretations of $[\alg{x}, C_{\phi}(\alg{x})]$ and $[\alg{x}, D_{\phi}(\alg{x})]$ in any Boolean category (including \Sets) will coincide with those of $[\alg{x}, \phi(\alg{x})]$ and $[\alg{x}, \neg \phi(\alg{x})]$, respectively, and that, moreover, $\theory^m$-models coincide with $\theory$-models in such categories. Also, since $S$ has cardinality less than $\kappa$, so does $\theory^m$. 

Now, if a sequent $\phi \vdash_{\alg{x}} \psi$ is valid in all $\theory$-models in \Sets, then $C_{\phi}\vdash_{\alg{x}}C_{\psi}$ will be valid in every $\theory^m$-model in \Sets, and therefore will be provable in $\theory^m$ by Proposition \ref{cohcomp}. Replace in this proof every subformula of the form $C_{\phi}(t_1, ..., t_{\alpha}, ...)$ by the corresponding substitution instance of $\phi(\alg{t}/\alg{x})$, and every subformula of the form $D_{\phi}(t_1, ..., t_{\alpha}, ...)$ by the corresponding substitution instance of $\neg \phi(\alg{t}/\alg{x})$. We claim that this way we will get a proof in $\theory$ of the sequent $\phi \vdash_{\alg{x}} \psi$ using the rules of $\kappa$-first-order systems. Indeed, the effect of the transformation just described on the axioms of $\theory^m$ produces either axioms of $\theory$ or sequents classically provable from $\theory$; finally, any instance of the transfinite transitivity rule used is, by Theorem \ref{equiv}, classically provable from the instances of classical distributivity and classical dependent choice.
\end{proof}

If one sticks to the transfinite transitivity rule, it is possible to sharpen these methods to prove completeness of a variety of classical systems. We will consider a variant of the instances of transfinite transitivity rule in which the conclusion is slightly modified. We call $TT_\gamma^B$ the rule:

\begin{mathpar}
\inferrule{\phi_{f} \vdash_{\mathbf{y}_{f}} \bigvee_{g \in \gamma^{\beta+1}, g|_{\beta}=f} \exists \mathbf{x}_{g} \phi_{g} \\ \beta<\gamma, f \in \gamma^{\beta} \\\\ \phi_{f} \dashv \vdash_{\mathbf{y}_{f}} \bigwedge_{\alpha<\beta}\phi_{f|_{\alpha}} \\ \beta < \gamma, \text{ limit }\beta, f \in \gamma^{\beta}}{\phi_{\emptyset} \vdash_{\mathbf{y}_{\emptyset}} \bigvee_{f \in B}  \exists_{\beta<\delta_f}\mathbf{x}_{f|_{\beta +1}} \bigwedge_{\beta<\delta_f}\phi_{f|_\beta}}
\end{mathpar}
\\
subject to the same provisos as the original rule, but where now $B \subseteq \gamma^{\gamma}$ consists of the minimal elements of a given bar over the tree $\gamma^{\gamma}$, and the $\delta_f$ are the levels of the corresponding $f \in B$. The proof of soundness of this rule in $\Sets$ is an easy modification of that of the original rule. If $TT_{\gamma}=\cup_{B}TT_{\gamma}^B$ and we denote by $\mathcal{L}_{\gamma^+, \gamma}(TT_{\gamma})$ Karp's classical system in which distributivity and dependent choice are replaced by the rules $TT_{\gamma}$, we have now:

\begin{thm}\label{karp2}
If $\gamma$ is regular and $\gamma^\alpha=\gamma$ for every $\alpha<\gamma$, theories of cardinality at most $\gamma$ in the classical system $\mathcal{L}_{\gamma^+, \gamma}(TT_{\gamma})$ are complete for \Sets-valued models.
\end{thm}

\begin{proof}
The proof follows the same lines as the proof for \ref{karp}, except that instead of relying on Proposition \ref{cohcomp}, we prove an analogous completeness theorem for theories with at most $\gamma$ axioms in the $(\gamma^+, \gamma, \gamma)$-coherent fragment of $\mathcal{L}_{\gamma^+, \gamma, \gamma}$. The system for this fragment is like that of the infinitary coherent fragment, with two differences: first, only disjunctions can be indexed by sets of cardinality $\gamma$, while the indexing sets of conjunctions and existential quantification is always less than $\gamma$ (this means that in the inductive definition of formulas of the type $\bigwedge_{i<\alpha} \phi_i$ and $\bigvee_{i<\alpha} \phi_i$ we need to make sure that $|\cup_{i<\alpha}FV(\phi_i)|<\gamma$, and in particular, the contexts of the formulas in the syntactic category have length always less than $\gamma$); second, the transfinite transitivity rule is replaced by the rule $TT_{\gamma}$, which is expressible in the fragment due to the hypothesis on $\gamma$. There is an obvious embedding from the syntactic category of this fragment into a sheaf topos, and one can show much as in the proof of Lemma \ref{shemb} that the embedding preserves the relevant structure.

We prove this completeness theorem much as in the proof of Proposition \ref{cohcomp}, but now the building covering families over each object $A$, used to construct the tree, witness if $A$ forces each of the $\gamma$-many subformulas of the axioms or of the valid sequent $\phi \vdash_x \psi$ (we call this set of subformulas $S$); that is, if the subformula $\eta$ is a (nonempty) disjunction $\bigvee_{i<\gamma}\phi_i(\boldsymbol{\beta})$ (resp. an existential formula $\exists_{\alpha<\gamma}\mathbf{x}_{\alpha}\psi(\mathbf{x}_0, ...,\mathbf{x}_{\alpha}, ..., \boldsymbol{\beta})$), and $A \Vdash \eta(\boldsymbol{\beta})$, we include in the set of coverings one of the form $l_j: C_j \to A$, where for each $j$ we have $C_j \Vdash \phi_{i_j}(\boldsymbol{\beta} l_j)$ for some $i_j<\gamma$ (resp. $C_j \Vdash \psi(\boldsymbol{\beta_0^j}, ..., \boldsymbol{\beta_{\alpha}^j}, ... \boldsymbol{\beta} l_j)$ for some $\boldsymbol{\beta_0^j}, ..., \boldsymbol{\beta_{\alpha}^j}, ...$). In case $\eta$ is $\bot$, or $\eta$ is a conjunctive subformula, or $A \nVdash \eta(\boldsymbol{\beta})$ we just consider the identity arrow as a cover. Thus although the tree has branching type $\gamma^+$, its height is $\gamma$.

It is enough to prove that every object in the sheaf model forcing the antecedent $\phi(\boldsymbol{\alpha})$ of the valid sequent $\phi \vdash_x \psi$ also forces the consequent $\psi(\boldsymbol{\alpha})$ for every tuple $\boldsymbol{\alpha}$ in the domain. We can thus consider a partial Beth model over the tree of height $\gamma$ so defined taking as the root of the tree an object forcing $\phi(\boldsymbol{\alpha})$, and the directed colimit $\mathbf{D_b}$ of all the underlying structures in the nodes of a cofinal branch $b$ of the tree. We then make it into a structure with the expected definition and prove that it is a model of the theory. For this, we prove that given any $\kappa$-coherent formula $\phi(x_0, ..., x_{\lambda}, ...) \in S$, we have $\mathbf{D_b} \vDash \phi(\overline{\alpha_0}, ..., \overline{\alpha_{\lambda}}, ...)$ if and only if for some node $n$ in $b$, the underlying structure $C_n$ satisfies $C_n \Vdash \phi(\alpha_0, ..., \alpha_{\lambda}, ...)$ for some representatives $\alpha_i$ of $\overline{\alpha_i}$ (the regularity of $\gamma$ is used in the definition of the structure in $\mathbf{D_b}$ and in the inductive step for conjunctions). Now, since any $(\gamma^+, \gamma, \gamma)$-coherent formula satisfied in the models given by the directed colimits of the underlying structures of the nodes along all cofinal branches, is forced at some node of every branch, an application of the categorical property corresponding to the rule $TT_{\gamma}$ proves that $\psi(\boldsymbol{\alpha})$ is forced at the roots, and the completeness of the $(\gamma^+, \gamma, \gamma)$-coherent fragment follows. 

Finally, for the classical Morleyization, we add to the axioms of $\theory^m$ also the axioms:

$$D_{\bigwedge_{i<{\gamma^+}}\phi_i}\dashv\vdash_{\alg{x}}\bigvee_{i<{\gamma^+}}D_{\phi_i},$$
\\
which, taking advantage of the classical relation between conjunctions and disjunctions, are able to code conjunctions with indexing set of size $\gamma$ into the $(\gamma^+, \gamma, \gamma)$-coherent fragment considered. Then we proceed as in Theorem \ref{karp} to finish the proof.
\end{proof}

\begin{rmk} 
 Theorem \ref{karp2} applies, for example, when $\gamma$ is inaccessible, or for any regular $\gamma$ as long as we assume the Generalized Continuum Hypothesis. It is also best possible in terms of the cardinality of theories for which one is able to derive completeness, since it is known from Jensen (see \cite{jensen}) that $V=L$ implies the existence of $\gamma^+$-Aronszajn trees, for which the theory of a cofinal branch (\emph{cf.} proof of Proposition \ref{converse}) is consistent but has no model.
\end{rmk}

\begin{cor}\label{karp3}
 (Karp) Countable theories in the classical system $\mathcal{L}_{\omega_1, \omega}$ are complete with respect to \Sets-valued models.
\end{cor}

\begin{proof}
 It is enough to apply Theorem \ref{karp2} and notice that $TT_{\omega}$ is actually provable in $\mathcal{L}_{\omega_1, \omega}$. To see this, we construct the proof tree as follows. Let 
 
$$C=\bigvee_{f \in B}  \exists_{\beta<\delta_f}\mathbf{x}_{f|_{\beta +1}} \bigwedge_{\beta<\delta_f}\phi_{f|_\beta}.$$
\\
To the premise:
 
$$\phi_{\emptyset} \vdash_{\mathbf{y}_{\emptyset}} \bigvee_{g \in \gamma} \exists \mathbf{x}_{g} \phi_{g}$$
\\
we append the sequent:

$$\bigvee_{g \in \gamma} \exists \mathbf{x}_{g} \phi_{g} \vdash_{\mathbf{y}_{\emptyset}} C$$
\\
to prepare an application of the cut rule and derive the conclusion we want. This latter sequent is in turn the conclusion of an application of the disjunction elimination rule from the set of sequents $\exists \mathbf{x}_{g} \phi_{g} \vdash_{\mathbf{y}_{\emptyset}} C$ for each $g \in \gamma$. To prove each of these, we use the following instance of the cut rule:

\begin{mathpar}
\inferrule{\exists \mathbf{x}_{g} \phi_{g} \vdash_{\mathbf{y}_{\emptyset}} \bigvee_{h \in \gamma^{2}, h|_{1}=g} \exists \mathbf{x}_{g} \mathbf{x}_{h} (\phi_{g} \wedge \phi_{h}) \\ \bigvee_{h \in \gamma^{2}, h|_{1}=g} \exists \mathbf{x}_{g} \mathbf{x}_{h} (\phi_{g} \wedge \phi_{h}) \vdash_{\mathbf{y}_{\emptyset}} C}{\exists \mathbf{x}_{g} \phi_{g} \vdash_{\mathbf{y}_{\emptyset}} C} 
\end{mathpar}
\\
and note that the first sequent on the left is derivable from a premise, and involves nodes of level $2$. This procedure is then repeated to the second sequent on the right, with further applications of the disjunction elimination rule and the cut rule, and we continue the tree proceeding upwards on every sequent containing $C$, until we reach all nodes in $B$. Since each node has finite height, the tree is well founded (i.e., it has no infinite ascending chains), and since each leaf is either a sequent provable from a premise or an instance of the disjunction introduction axiom, the tree provides the proof we wanted.
\end{proof}

\subsection{Makkai's theorem}

Consider the $\kappa$-regular fragment; the syntactic category of a theory over this fragment is a $\kappa$-regular category. If we consider the topos of sheaves over this category with the regular coverage given by single epimorphisms, the coverage is subcanonical and the topos is a conservative sheaf model for the theory, as can be proved analogously to Lemma \ref{shemb}. We have now:

\begin{proposition}
 Let $\kappa$ be a regular cardinal. Then any $\kappa$-regular theory of cardinality at most $\kappa$ has a linear partial Beth model of height $\kappa$.
\end{proposition}

\begin{proof}
The proof follows the same pattern as the proof of Theorem \ref{shcomp}, but is simplified in this case because the building covering families over each object $A$, used to construct the tree, consist of one element, thereby guaranteeing that the tree constructed will be linear. More especifically, the cover over an object $A$ witness if $A$ forces each of the less than $\kappa$ many subformulas of the axioms; that is, if the subformula $\eta$ is an existential formula $\exists_{\alpha<\gamma}\mathbf{x}_{\alpha}\psi(\mathbf{x}_0, ...,\mathbf{x}_{\alpha}, ..., \boldsymbol{\beta})$), and $A \Vdash \eta(\boldsymbol{\beta})$, we include in the set of coverings one of the form $l: C \to A$, where we have $C \Vdash \psi(\boldsymbol{\beta_0}, ..., \boldsymbol{\beta_{\alpha}}, ... \boldsymbol{\beta} l)$ for some $\boldsymbol{\beta_0}, ..., \boldsymbol{\beta_{\alpha}}, ...$. In case $\eta$ is a conjunctive formula or $A \nVdash \eta(\boldsymbol{\beta})$ we just consider the identity arrow as a cover. To guarantee that the tree has height $\kappa$, we add to the set of covers over an object identity arrows to the set of building covering families over each object $A$ until it has cardinality $\kappa$. 
\end{proof}
 
As a consequence, we immediately get:

\begin{proposition}\label{regbcomp}
If $\kappa$ is a regular cardinal, $\kappa$-regular theories of cardinality less than $\kappa$ are complete with respect to \Sets-valued models. 
\end{proposition}

\begin{proof} 
Again, the proof is similar to that of Proposition \ref{cohcomp}. It is enough to prove that every object in the sheaf model forcing the antecedent $\phi(\boldsymbol{\alpha})$ of a sequent $\phi \vdash_x \psi$ also forces the consequent $\psi(\boldsymbol{\alpha})$ for every tuple $\boldsymbol{\alpha}$ in the domain. We can thus consider a partial Beth model over a linear tree as above but taking instead as the root of the tree an object forcing $\phi(\boldsymbol{\alpha})$, and including in the set of subformulas of the axioms also subformulas in the valid sequent $\phi \vdash_x \psi$ (we call this set of subformulas $S$). Consider the directed colimit $\mathbf{D}$ of all the underlying structures in the nodes of the tree; we then make it into a structure with the expected definition and prove that it is a model of the theory. For this, we prove that given any $\kappa$-regular formula $\phi(x_0, ..., x_{\lambda}, ...) \in S$, we have $\mathbf{D} \vDash \phi(\overline{\alpha_0}, ..., \overline{\alpha_{\lambda}}, ...)$ if and only if for some node $n$ in the linear tree, the underlying structure $C_n$ satisfies $C_n \Vdash \phi(\alpha_0, ..., \alpha_{\lambda}, ...)$ for some representatives $\alpha_i$ of $\overline{\alpha_i}$ (the regularity of $\kappa$ is used in the definition of the structure in $\mathbf{D}$ and in the inductive step for conjunctions). Finally, since any $\kappa$-regular formula satisfied in the models given by the directed colimits of the underlying structures of the nodes in the linear trees, is forced at their roots (as can be seen by an application of the dependent choice property), the result follows. 
\end{proof}

Proposition \ref{regbcomp} can be improved by removing the restriction on the cardinality of the theories in question, which leads us to a result of \cite{makkai}:

\begin{thm}\label{sregbcomp}
(Makkai) If $\kappa$ is a regular cardinal, $\kappa$-regular theories are complete with respect to \Sets-valued models.
\end{thm}

\begin{proof}
 Suppose that the sequent $\phi \vdash_{\mathbf{x}} \psi$ is valid in every model of a certain theory but not provable. Then it is not provable in any subtheory of cardinality less than $\kappa$. Therefore, if we add to the language new constants $\mathbf{c}$ and axioms $\top \vdash \phi(\mathbf{c})$ and $\psi(\mathbf{c}) \vdash \bot$, any subtheory of cardinality less than $\kappa$ together with these two new axioms has, by Proposition \ref{regbcomp}, a model. Now we can build a model for  the whole theory by an usual reduced product construction using each of the models so far available, but where we only use a $\kappa$-complete filter instead of an ultrafilter (which will be enough for our purposes since we do not have disjunction among the connectives). To do so, suppose the cardinality of the theory is $\lambda$, and define $P_{\kappa}(\lambda)=\{x \subseteq \lambda: |x|<\kappa\}$. For each $x \in P_{\kappa}(\lambda)$ let $\mathcal{M}_x$ be a model for the axioms in $x$. Now notice that the set:
 
 $$\{\{x \in P_{\kappa}(\lambda): y \subseteq x\}: y \in P_{\kappa}(\lambda)\}$$
 \\
 generates a $\kappa$-complete filter $\mathcal{F}$ on $P_{\kappa}(\lambda)$ by the regularity of $\kappa$. Then we can define a model for the whole theory as the reduced product $\Pi_{x \in P_{\kappa}(\lambda)}\mathcal{M}_x / \mathcal{F}$. The usual proof of \L{}o\'s theorem can be adapted for the $\kappa$-regular case with the filter $\mathcal{F}$ (the use of an ultrafilter is in the inductive step showing that disjunctive subformulas satisfy the theorem, and so it is not needed here). This provides a model for the original theory where $\phi \vdash_{\mathbf{x}} \psi$ is not valid.
\end{proof}

\begin{rmk}
In the particular case when $\kappa=\omega$, Theorem \ref{sregbcomp} reduces to the completeness theorem for regular logic. As it happens with coherent logic, the proof given is thus an alternative path to the usual categorical treatment of Henkinization from \cite{johnstone}.
\end{rmk}

\subsection{Fourman-Grayson theorem}

In \cite{fourman}, the completeness of countably axiomatized geometric propositional theories was proven. Here we check that the general case, where we include also existential quantification, is deduced from the same techniques employed so far in the completeness theorems, and at this point it only needs to be outlined. We have:

\begin{thm}\label{fg}
Countable geometric theories are complete with respect to \Sets-valued models. 
\end{thm}

\begin{proof}
The proof follows the same lines as the proof for Proposition \ref{cohcomp}, except that building covering families over each object $A$, used to construct the tree, witness if $A$ forces each of the (countably many) antecedents and consequents of the axioms or of the valid sequent $\phi \vdash_{\mathbf{x}} \psi$; that is, if an antecedent or consequent $\eta$ is a (nonempty) disjunction of the form $\bigvee_{i<\gamma}\exists \mathbf{x}_{0}...\mathbf{x}_{n_i} \psi_i(\mathbf{x}_0, ...,\mathbf{x}_{n_i}, \boldsymbol{\beta})$, and $A \Vdash \eta(\boldsymbol{\beta})$, we include in the set of coverings one of the form $l_j: C_j \to A$, where for each $j$ we have $C_j \Vdash \psi_{i_j}(\boldsymbol{\beta_0^j}, ..., \boldsymbol{\beta_{n_{i_j}}^j}, \boldsymbol{\beta} l_j)$ for some $i_j$ and some $\boldsymbol{\beta_0^j}, ..., \boldsymbol{\beta_{n_{i_j}}^j}$. In case $\eta$ is $\bot$, or $\eta$ is a conjunctive subformula, or $A \nVdash \eta(\boldsymbol{\beta})$ we just consider the identity arrow as a cover. This guarantees that the tree that we will use to build the partial Beth model will have countable height (although its branching type could be quite high), and the structures built as the filtered colimits of underlying structures in a cofinal branch will be (possibly exploding) positive geometric models of the theory. Finally, since any geometric formula satisfied in the models given by the directed colimits of the underlying structures of the nodes along all cofinal branches, is forced at some node of every branch, an application of the categorical property corresponding to the rule $TT_{\omega}$ (itself provable in geometric logic, with the same proof sketched in the proof of Corollary \ref{karp3}) proves that the consequent of the valid sequent is forced at the roots, and the completeness result follows.
\end{proof}

\begin{rmk}
Theorem \ref{fg} is best possible in terms of the cardinality of the theories. Indeed, given an $\omega_1$-Aronszajn tree, the theory of a cofinal branch there (\emph{cf.} proof of Proposition \ref{converse}), is obviously geometric and of cardinality $\omega_1$, but although consistent, it has no model.
\end{rmk}

\subsection{Future work}

Besides finding the exact strength of Heyting cardinals and the disjunction and existence properties within the large cardinal hierarchy, another aspect which was not considered here is the question of whether conceptual completeness results could be established, or to what extent the category of models of a given theory determines the theory, up to $\kappa$-pretopos completion. Conceptual completeness theorems were obtained in \cite{mr}, and Makkai provided an even stronger result in \cite{makkai2}, by proving that the category of models of a theory could be endowed with certain structure whose study would allow to recover the theory. Awodey and Forssell provided in \cite{af} a different approach by considering a topological structure on the category of models. Neither of these approaches has been generalized to the infinitary first-order case, although Makkai gave in \cite{makkai} an answer for the infinitary regular fragment. It is therefore a natural step first to try to obtain conceptual completeness theorems for the infinitary first-order case, and second, to identify which type of structure could be given to the category of models to be able to recover the theory, possibly using some large cardinal assumptions.

\subsection{Acknowledgements} The main ideas of this work have been presented in my PhD thesis at Stockholm University, and my period there was funded by the Swedish Research Council. I am indebted to all people that encouraged and supported me during my studies, in particular to my advisors Erik Palmgren and Henrik Forssell, as well as to Peter Lumsdaine for useful discussions about the subject.

\bibliographystyle{amsalpha}

\renewcommand{\bibname}{References} 

\bibliography{references}



\end{document}